\documentclass[11pt]{amsart}

\usepackage{amsmath}
\usepackage{amsfonts}
\usepackage{amssymb}
\usepackage{amsthm}
\usepackage{graphics}
\usepackage{amscd}
\usepackage{graphicx,epsfig}
\usepackage{color}
\usepackage[all]{xy}
\usepackage{url}

\DeclareMathOperator{\tr}{tr}

\renewcommand{\epsilon}{\varepsilon}

\newcommand{\boF}{\mathcal{F}}

\newcommand{\boM}{\mathcal{M}}

\newcommand{\boA}{\mathcal{A}}

\newcommand{\boB}{\mathcal{B}}

\newcommand{\boL}{\mathcal{L}}

\newcommand{\boN}{\mathcal{N}}
\newcommand{\boZ}{\mathcal{Z}}

\newcommand{\MM}{\mathbf{M}}

\newcommand{\R}{\mathbb{R}}

\newcommand{\Z}{\mathbb{Z}}

\renewcommand{\S}{\mathbb{S}}

\newcommand{\eps}{\varepsilon}

\newcommand{\Ome}{\Omega}

\newtheorem{defn}{Definition}
\newtheorem{thm}{Theorem}
\newtheorem{cor}{Corollary}
\newtheorem{prop}{Proposition}
\newtheorem{lem}{Lemma}
\newcommand{\inter}[1]{\overset{\circ}{#1}}
\newcommand{\barre}[1]{\overline{#1}}
\renewcommand{\phi}{\varphi}

\newtheorem*{thm*}{Theorem}

\newcounter{remark}

\newcounter{case}

\newcounter{construction}

\newcounter{fact}

\newcommand{\Hbf}{\mathbf{H}}
\newcommand{\Lbf}{\mathbf{L}}
\newcommand{\pr}{\partial}

\newcommand{\nb}{\barre \nabla}
\newcommand{\Za}{Z^{(a)}}
\newcommand{\Iso}{\mathfrak{Is}}
\DeclareMathOperator{\Span}{span}
\DeclareMathOperator{\spt}{spt}
\DeclareMathOperator{\ric}{Ric}

\title{Rigidity of riemannian manifolds containing an equator}
\author{Laurent Mazet}
\address{Institut Denis Poisson, CNRS UMR 7013, Universit\'e de Tours, 
Universit\'e d'Orl\'eans, Parc de Grandmont, 37200 Tours, France}
\email{laurent.mazet@univ-tours.fr}

\thanks{The authors was partially supported by the ANR-19-CE40-0014 grant.}

\begin{document}

\maketitle

\begin{abstract}
In this paper, we prove that a Riemannian $n$-manifold $M$ with sectional 
curvature bounded above by $1$ that contains a minimal $2$-sphere of area 
$4\pi$ 
which has index at least $n-2$ has constant sectional curvature $1$. The proof 
uses the construction of ancient mean curvature flows that flow out of a 
minimal submanifold. As a consequence we also prove a rigidity result for the 
Simon-Smith minimal spheres.
\end{abstract}

\section{Introduction}

Let $g$ be a complete Riemannian metric on the $2$-sphere $\S^2$. If its 
sectional curvature is between $0$ and $1$, it is known that any closed 
geodesic on $(\S^2,g)$ has length at least $2\pi$ \cite{Pog2}. Moreover if such 
a closed geodesic has length $2\pi$, $(\S^2,g)$ is isometric to the unit 
$2$-sphere $\S_1^2=\{p\in\R^3\mid \|p\|=1\}$ with the induce metric. The proof 
of this result is given in \cite{AnHo} where the authors attribute the theorem 
to E.~Calabi.

So a question is what happens in higher dimension. In dimension $3$, one can 
replace geodesics by minimal $2$-sphere. Actually one can prove that, if the 
sectional curvature is bounded above by $1$, any minimal $2$-sphere has area at 
least 
$4\pi$. In \cite{MaRo}, H.~Rosenberg and the author study the equality case. If 
$(M,g)$ is a Riemannian $3$-manifold with sectional curvature $0\le K\le 1$ 
that contains a minimal $2$-sphere of area $4\pi$, they prove that the 
universal cover of $M$ is isometric to the unit $3$-sphere $\S_1^3$ or the 
product $\S_1^2\times \R$.

One purpose of this paper is to investigate generalizations of this result to 
higher dimensions. Actually if $(M,g)$ is a Riemannian $n$-manifold with 
sectional curvature $K\le 1$, we still have that the area of a minimal 
$2$-sphere is at least $4\pi$. So what can be said in the equality case ?

A model of this situation is an equatorial $2$-sphere in the unit $n$-sphere 
$\S_1^n$. So one could expect that under some extra hypotheses this is the only
example.

If $\Sigma$ is a minimal $m$-submanifold in $M$, $\Sigma$ is critical for the 
volume functional. The stability of this critical point is given by the Jacobi 
operator which is a self-adjoint second order elliptic operator that acts on 
sections of the normal bundle to $\Sigma$. As a critical point, the index of 
$\Sigma$ is given by the number of negative eigenvalues of this operator. In 
the case of an equatorial $2$-sphere $S$ in $\S_1^n$, the index of $S$ is $n-2$.

The first main result of the paper is a rigidity result under such an 
instability hypothesis.

\begin{thm*}
Let $M$ be a Riemannian $n\ge 3$-manifold whose sectional curvature is bounded 
above by $1$. Let us assume that $M$ contains an immersed minimal $2$-sphere of 
area $4\pi$ which has index at least $n-2$. Then the universal cover of $M$ is 
isometric to the unit sphere $\S_1^n$.
\end{thm*}

Let us notice that the instability hypothesis can be replaced by an other 
version.

\begin{defn}
Let $\Sigma$ be a minimal submanifold in $(M,g)$. We say that $\Sigma$ is 
unstable in any parallel directions if the restriction of the Jacobi operator 
to any parallel sub-bundle of the normal bundle to $\Sigma$ has index at least 
$1$.
\end{defn}

The above theorem gives an answer to a question that arises from a result in 
\cite{AnHo}. In \cite[Corollary 5.11]{AnHo}, L.~Anderson and R.~Howard prove 
that a Riemannian $n$-manifold $M$ with sectional curvature below $1$  
containing 
isometrically a neighborhood of the equator $\S_1^{n-1}$ in $\S_1^{n}$ is 
isometric to $\S_1^n$. The hypothesis that a whole neighborhood of the equator 
belongs to $M$ seems strong and the question is to find a weaker hypothesis. 
Actually our main result gives an infinitesimal 
version of Anderson-Howard result. If $M$, with $K\le 1$, contains a 
totally geodesic hypersurface isometric to $\S_1^{n-1}$ that is unstable as a 
minimal hypersurface, then $M$ is isometric to $\S_1^n$. The idea is that the 
totally geodesic $\S_1^{n-1}$ contains a minimal $2$-sphere of area $4\pi$ and 
index at least $n-2$. Actually in the same spirit as Anderson-Howard theorem, 
there is a 
result by D.~Panov and A.~Petrunin \cite[Theorem 1.4]{PaPe} with a weaker
hypothesis: if $S$ is an equatorial $2$-sphere in $\S_1^n$ and $S^+$ denotes an 
hemisphere of $S$, Panov and Petrunin need just that $M$ contains isometrically 
a neighborhood of $S^+$ in $\S_1^n$. 

The proof of the main theorem uses ideas that already appear in \cite{MaRo}:  
if $S$ is an immersed $2$-sphere we define the $F$ functional by 
$F(S)=\boA(S)+\int_S\|\vec H\|^2-4\pi$ where $\boA(S)$ is the area of $S$ and 
$\vec 
H$ 
is 
the mean curvature vector of $S$. Actually if $F(S)$ vanishes, $S$ is 
totally umbilical and we obtain some information on the sectional curvature of 
$M$ along $S$. So if $S_0$ is the minimal $2$-sphere given by the statement of 
the theorem $F(S_0)=0$. The idea is to explore the geometry of $M$ by computing 
$F(S_t)$ along a deformation $\{S_t\}_t$ of $S_0$. One of the novelties is the 
construction of the family $\{S_t\}_t$. Actually we produce $\{S_t\}$ as a mean 
curvature flow that flows out of $S_0$. More precisely, we construct non 
trivial ancient solutions $\{S_t\}_{t\in(-\infty,b)}$ of the mean curvature 
flow 
such that, as $t\to -\infty$, $S_t$ converges to $S_0$.

The idea is that the eigen-sections of the Jacobi operator associated to the 
first eigenvalue give directions in which such an ancient mean curvature flow 
can be initiated. A similar idea appear in the work of K.~Choi and 
C.~Mantoulidis~\cite{ChMa} where they construct ancient mean curvature flows 
"tangent" to the eigenspaces with negative eigenvalues. Then they prove several 
uniqueness results for ancient mean curvature flow in $\S_1^n$.
An other example is \cite{MrPa} where A.~Mramor and A.~Payne produce an 
eternal solution of the mean curvature flow that flows out of the catenoid.  
Let us 
notice that good introductions to the study of high codimension mean 
curvature flow can be found in the paper of K.~Smoczyk \cite{Smo} and the PhD 
thesis of C.~Baker \cite{Bak}.

The index hypothesis in the above theorem can appear very particular. Actually 
there are situations where it is quite natural. As critical points of the area 
functional, minimal hypersurfaces can be produced by a minimization process. 
However one have to consider a non-trivial class of hypersurfaces to produce a 
non-trivial critical point. So in order to solve this difficulty, a Morse 
theoretical approach has been developed. In \cite{Smi}, F.~Smith is able to  
contruct minimal $2$-spheres in any riemannian $(\S^3,g)$. It's proof is based 
on the following ideas. Let $\Lambda$ be the set of paths 
$\{\sigma_t\}_{t\in[-1,1]}$ in the space of $2$-spheres in $\S^3$ that 
sweeps out $\S^3$ (see precise definitions and statements in 
Section~\ref{sec:minmax}). Then he considers the quantity
\[
W(\S^3,g)=\inf_{\{\sigma_t\}\in\Lambda}\max_{t\in[-1,1]}\boA(\sigma_t)
\]
called the Simon-Smith width of $(\S^3,g)$. 

First this quantity is positive and Smith proves that it is realized by the 
area of a finite collection of minimal spheres. Besides it is reasonable to 
think that the index of these collection of minimal spheres is $1$. 
F.~C.~Marques and 
A.~Neves \cite{MaNe6} proved the upper-bound by $1$. So the  
second main result of this paper is
\begin{thm*}
Let $(\S^3,g)$ be a Riemannian $3$-sphere whose sectional curvature is bounded 
above by $1$. Then $W(\S^3,g)\ge 4\pi$ and, if $W(\S^3,g)=4\pi$, then 
$(\S^3,g)$ is isometric to $\S_1^3$.
\end{thm*}

If one knows that $W(\S^3,g)$ is realized by an index $1$ minimal $2$-sphere 
the above result is a direct consequence of our first rigidity result. So the 
difficulty is to deal with the case where $W(\S^3,g)$ is realized by a  
$2$-sphere of index $0$. Actually one can think about the following example: 
the cylinder $\S_1^2\times[-1,1]$ capped by two hemispheres ${\S_1^3}^+$ (see 
Figure \ref{fig:fig1}).
\begin{figure}
\resizebox{0.6\linewidth}{!}{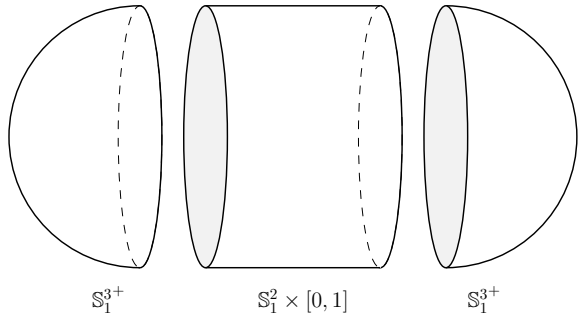}
\caption{a $C^{1,1}$ Riemannian metric with width $4\pi$}\label{fig:fig1}
\end{figure}
This defines a $C^{1,1}$ Riemannian metric $\bar g$ on $\S^3$ whose sectional 
curvature is bounded above by $1$ in any reasonable weak sense. Its Simon-Smith 
width is $4\pi$ so this implies that the above result is false for a weak sense 
of sectional curvature. Actually the above example is exactly the type of 
situation we have to consider in the proof of the above theorem: we prove that 
the Simon-Smith width is realized by a minimal $2$-sphere which is not almost 
stable. Moreover $\bar g$ can be smoothed to produce 
sequences $(g_n)$ of smooth Riemannian metrics with sectional curvature bounded 
above by $1$, $W(\S^3,g_n)\to 4\pi$ and $g_n\to \bar g$. So the $g_n$ are far 
of the round metric of $\S^3_1$. This implies that the above rigidity result is 
not stable.

Let us also notice studying the Simon-Smith width for reversed curvature 
inequalities has been done by F.~C.~Marques and A.~Neves. This time just a 
control on the Ricci and scalar curvature is 
assumed.
\begin{thm*}[Marques, Neves \cite{MaNe3}]
Let $(\S^3,g)$ be a Riemannian $3$-sphere with positive Ricci curvature and 
scalar curvature $R\ge 6$. Then $W(\S^3,g)\le 4\pi$ and, if $W(\S^3,g)=4\pi$, 
then $(\S^3,g)$ is isometric to $\S_1^3$.
\end{thm*}

The paper is organized as follows. In Section~\ref{sec:geometry}, we recall 
some basic formulas and definitions of submanifold geometry.
Section~\ref{sec:mcf} is devoted to the construction of ancient solutions of 
the mean curvature flow (Theorem~\ref{thm:ancientmcf}). In 
Section~\ref{sec:rigidity}, we prove our first rigidity result 
(Theorem~\ref{th:rigidity}) and its Corollary concerning manifolds containing 
an equator of $\S_1^n$. Section~\ref{sec:minmax} is devoted to the study of the 
Simon-Smith width and the proof of the second rigidity result 
(Theorem~\ref{th:rigidityminmax}). In Appendix~\ref{app:schauder}, we prove a 
Schauder type estimate used in the proof of Theorem~\ref{thm:ancientmcf}.

\bigskip

The author would like to thank C. Mantoulidis for discussions about his result 
and A.~Petrunin for pointing him out reference \cite{PaPe}.

\section{Geometry of submanifolds}\label{sec:geometry}

In this section we recall some classical notations and formulas concerning the 
geometry of submanifolds.

Let $(M,g)$ be a Riemannian manifold of dimension $n$ and $\Sigma$ a manifold 
of dimension $m$. If $F_0:\Sigma\to 
M$ is an immersion, we can consider the induced Riemannian metric $g_0=F_0^*g$ 
on 
$\Sigma$ making $F_0$ a local isometry. In the paper, we often identify 
$\Sigma$ 
with its image $\Sigma_0=F_0(\Sigma)$ at least locally where $F_0$ is an 
embedding: for example, we often identify $T_p\Sigma$ with 
$(F_0)_*(T_p\Sigma)\subset 
T_{F_0(p)}M$.   

If $\nabla$ and $\nabla^0$ are respectively the Levi-Civita connections on $M$ 
and $\Sigma$, we can define the second fundamental form on $\Sigma$ by 
\[
B_p(v,w)=\nabla_vw-\nabla^0_vw \in N_p\Sigma
\]
where $v,w\in T_p\Sigma$ and $N_p\Sigma$ is the normal subspace to $\Sigma$ at 
$p$. 

The mean curvature vector of $\Sigma$ is then 
\[
\vec H(p)=\frac1m \tr_{T_p\Sigma}B_p\in N_p\Sigma
\]
where $\tr_P$ denotes the trace operator on the subspace $P$.
We define $\inter B_p=B_p-\vec H(p) g_0$ the traceless part of the second 
fundamental form. We recall that the normal bundle $N\Sigma$ inherits from $g$ 
and $\nabla$ a 
normal connection $\nabla^\perp$. 

Let $(F_t)_t$ be a smooth family of immersion of $\Sigma$ and define the 
vectorfield $X=\frac{d}{dt}{\phi_t}_{|t=0}$ along $F_0(\Sigma)$ and let 
$\Sigma_t=F_t(\Sigma)$. We have a 
family of metrics $g_t=F_t^*g$ defined on $\Sigma$ with associated 
volume 
measure 
$d\sigma_t$. If, $X$ is orthogonal to $\Sigma$, it is well 
known that, for any function $f$ on $\Sigma$:
\[
\frac{d}{dt}_{|t=0}\int_\Sigma f d\sigma_{g_t}=-\int_\Sigma (X,m\vec H)f 
d\sigma_0
\]
So if $\Sigma$ is critical with respect to the $m$-volume functional $\boA$, we 
have 
$\vec 
H=0$ along $\Sigma$: $\Sigma$ is minimal.

We are interested in understanding how the mean curvature vector $\vec H$ is 
deformed along the family $F_t$. So let us denote by $\vec H_t(p)$ the mean 
curvature vector of $\Sigma_t$ at $F_t(p)$.  

\begin{lem}
If $X$ is normal to $\Sigma$, we have
\[
\frac{D}{dt}{m\vec H_t}_{|t=0}=\Delta^\perp X+(R(e_i,X)e_i)^\perp + 
(X,B(e_i,e_j))B(e_i,e_j)-(m\vec H_0, \nabla_{e_i}X)e_i
\]
with the convention that summations are made over repeated indices, 
$(e_1,\dots,e_m)$ is an orthonormal frame of $\Sigma_0$, $R$ is the Riemann 
curvature tensor associated to $g$ and $\Delta^\perp$ 
denotes the Laplacian operator acting on normal sections: $\Delta^\perp 
X=\tr {\nabla^\perp}^2 X$.  
\end{lem}
\begin{proof}
Let $E_1,\dots,E_m$ be an orthonormal frame on $(\Sigma,g_0)$ and consider 
along $F_t(p)$ the frame $e_i=(F_t)_*(E_i)$. We assume that 
$\nabla^0_{e_i}e_j=0$ for any $i,j$ at $\bar p$ where the computation is made. 
Let us denote $g_{ij}=(e_i,e_j)$ and 
$(g^{ij})$ the inverse matrix. We have
\[
mH_t= g^{ij}(\nabla_{e_i}e_j)^\perp
\]
where $Y^\perp$ denotes the orthogonal projection to $N\Sigma_t$. At $t=0$, we 
have $\frac{D}{dt}e_i=\nabla_{e_i}X$.

Notice that ${g_{ij}}_{|t=0}=\delta_i^j$, so at $t=0$:
\begin{align*}
\frac{d}{dt}g^{ij}=-\frac{d}{dt}g_{ij}&=-(\frac{D}{dt}e_i,e_j)- 
(e_i,\frac{D}{dt}e_j)\\
&=-(\nabla_{e_i}X,e_j)-(e_i,\nabla_{e_j}X)
=2(X,B(e_i,e_j))
\end{align*}

If $Y$ is a vector field along $F_t(p)$. We have 
$Y^\perp=Y-g^{ij}(Y,e_j)e_i$, 
so at $t=0$:
\begin{align*}
\frac{D}{dt}Y^\perp&=\left(\frac{D}{dt}Y\right)^\perp +((\nabla_{e_i}X,e_j)+ 
(e_i,\nabla_{e_j}X)) (Y,e_j)e_i\\
&\qquad-(Y,\frac{D}{dt}e_i)e_i-(Y,e_i)\frac{D}{dt}e_i\\
&=\left(\frac{D}{dt}Y\right)^\perp 
+(\nabla_{e_i}X,Y^\top)e_i+(Y,e_j)(\nabla_{e_j}X)^\top\\
&\qquad-(Y,\nabla_{e_i} X)e_i-(Y,e_i)\nabla_{e_i}X\\
&=\left(\frac{D}{dt}Y\right)^\perp-\sum_i(Y^\perp,\nabla_{e_i}X)e_i-(Y,e_i)(\nabla_{e_i}X)^\perp
\end{align*}
where $Z^\top$ denotes the tangential part of $Z$.

We also have
\begin{align*}
\frac{D}{dt}\nabla_{e_i}e_i&=R(e_i,X)e_i+\nabla_{e_i}\frac{D}{dt}e_i\\
&=R(e_i,X)e_i+\nabla_{e_i}\nabla_{e_i}X\\
\end{align*}
So combining all the above computations at $\bar p$, we obtain
\[
\frac{D}{dt}mH_t=2(X,B(e_i,e_j))B(e_i,e_j)+(R(e_i,X)e_i)^\perp+ 
(\nabla_{e_i}\nabla_{e_i}X)^\perp-(mH_0,\nabla_{e_i}X)e_i
\]
Since $ (\nabla_{e_i}\nabla_{e_i}X)^\perp= 
\nabla_{e_i}^\perp\nabla_{e_i}^\perp X-(X,B(e_i,e_j))B(e_i,e_j)$ we 
finally have
\[
\frac{D}{dt}mH_t=\Delta^\perp 
X+(X,B(e_i,e_j))B(e_i,e_j)+(R(e_i,X)e_i)^\perp-(mH_0,\nabla_{e_i}X)e_i
\]
\end{proof}

As a consequence, if $\Sigma_0$ is minimal, the second derivative of the 
$m$-volume of $\Sigma_t=F_t(\Sigma)$ is given by
\begin{align*}
\frac{d^2}{dt^2}\boA(\Sigma_t)_{|t=0}&=-\int_\Sigma(X,\Delta^\perp 
X+(X,B(e_i,e_j))B(e_i,e_j)+(R(e_i,X)e_i))d\sigma_0\\
&=\int_\Sigma\|\nabla^\perp X\|^2-(R(e_i,X)e_i,X)-(X,B(e_i,e_j)^2)d\sigma_0^2\\
&=Q_\Sigma(X,X)
\end{align*}
So $Q_\Sigma$ is a quadratic form acting on sections of the normal bundle 
$N\Sigma$. It is attached to the Jacobi operator acting on normal sections:
\[
LX=\Delta^\perp X +(R(e_i,X)e_i)^\perp+(X,B(e_i,e_j))B(e_i,e_j)
\]

This operator is elliptic and self-adjoint. It has a spectrum $\lambda_0\le 
\lambda_1\le \cdots$. If $\lambda_0<0$, $\Sigma$ is called unstable. The index 
of $L$ (the number of negative eigenvalues) is called the index of $\Sigma$.

\section{Ancient solutions of the mean curvature flow}\label{sec:mcf}

\subsection{The mean curvature flow}
First let us recall some basics of the mean curvature flow and state our main 
result. For a good introduction to the high co-dimension case, one can have a 
look to Smoczyk's paper \cite{Smo}.

Let $(M,g)$ be a Riemannian manifold and $\Sigma$ a $m$-manifold.  Let $F: 
\Sigma\times I\to M$ ($I$ an interval) be a smooth map such that 
$F_t=F(\cdot,t)$ is an immersion for any $t$. We say that $F_t(\Sigma)$ evolve 
by mean curvature flow if for any $p\in \Sigma$ and $t\in I$
\[\label{eq:mcf}\tag{MCF}
\frac{dF}{dt}(p,t)=m\vec{H}(p,t)
\]
where $\vec{H}\in T_{F_t(p)}M$ is the mean curvature vector of $F_t(\Sigma)$ at 
$F_t(p)$. 

For example, if $F_0(\Sigma)$ is a minimal submanifold, then $F_t=F_0$ for 
$t\in I$ is a solution of the mean curvature flow: minimal submanifolds are 
fixed points of the mean curvature flow.

Our aim is to produce solutions that flow out of a minimal 
surface. More precisely, we construct non constant ancient solutions of the 
mean 
curvature flow (\textit{i.e.} defined on a time interval $(-\infty,b)$) such 
that, as 
$t\to-\infty$, $F_t(\Sigma)$ converges to a minimal surface.

It is well known that one difficulty in the solvability of \eqref{eq:mcf} is 
the invariance under the diffeomorphism group which causes a lack of 
parabolicity of the system. One solution to this difficulty consists in adding 
a tangential component to the time derivative of $F$ which has no impact on the 
geometric evolution.

Let us explain such a solution. Let $\Sigma$ be an immersed closed 
submanifold in $M$. Let $N\Sigma$ denote 
the normal bundle to $\Sigma$. Then we can consider the map 
\[
\Phi : \begin{array}{ccc}
N\Sigma &\longrightarrow& M\\
(p,v)&\longmapsto& \exp_p(v)
\end{array}
\]
For $\eps>0$, let us denote $N\Sigma^\eps=\{(p,v)\in N\Sigma\mid \|v\|<\eps\}$. 
If $\eps$ is small enough, the restriction of $\Phi$ to $N\Sigma^\eps$ is an 
immersion so the metric $g$ can be lifted to $h=\Phi^* g$ on $N\Sigma^\eps$. 
Now studying immersed submanifolds close to $\Sigma$ consists in looking at 
sections of $N\Sigma^\eps$ close to $0$. Actually one can extend the Riemannian 
metric $h$ to the whole $N\Sigma$ and just look at sections close to $0$. 

So the general setting we have to consider is the following. Let $E$ be a 
vector bundle over a closed manifold $\Sigma$ and consider $g$ a 
Riemannian metric on the manifold $E$. We say that $E$ is a normal bundle if 
the fibers are orthogonal to $\Sigma_0$ the image of the $0$ 
section. If $E$ is a normal bundle there is a natural identification between 
$E$ and the normal bundle to $\Sigma_0$. So as a normal bundle, $E$ inherits a 
bundle metric $g^\perp$ and a connection $\nabla^\perp$.

If $U$ is a section 
of $E$, $U(\Sigma)$ is a submanifold in $E$. Then sections are a 
particular way to parametrize submanifolds in $E$. Let $U$ be a section of 
$E$ and 
$p\in\Sigma$. Since $U$ is a section, the tangent space $T_{U(p)}E$ 
splits as $T_pU\oplus T_{U(p)}E_p$ where $E_p$ is the fiber of 
$E$ over $p$. Moreover there is a natural identification of 
$T_{U(p)}E_p$ with $E_p$. So for any $Y\in T_{U(p)}E$, one 
can define $Y^\sharp$ the projection of $Y$ to $E_p$ parallel to $T_pU$.

With this type of notation, we can define the bundle mean 
curvature flow in the following way: let $U:\Sigma\times I\to E$ a 
smooth map such that $U_t=U(\cdot,t)$ is a section of $E$, we say 
that $(U_t)_{t\in I}$ evolves by bundle mean curvature flow if for any $p\in 
\Sigma$ and $t\in I$ 
\begin{equation}\label{eq:bmcf}
\frac{dU}{dt}(p,t)=(m\vec H(U_t,p))^\sharp
\end{equation}
where $\vec{H}\in T_{U_t(p)}E$ is the mean curvature vector of the 
graph of $U_t$ at $U_t(p)$.

$(\vec H(U_t,p))^\sharp$ is equal to $\vec 
H(U_t,p)$ plus a tangent vector to $U_t(\Sigma)$. So solutions to 
\eqref{eq:bmcf} give rise to 
solutions to the mean curvature flow \eqref{eq:mcf} after a reparametrization.

Let us define the operator $\Hbf: \Gamma (E)\to \Gamma (E)$ by 
$\Hbf(U)(p)=(m\vec H(U_t,p))^\sharp$. $\Hbf$ is a smooth quasilinear elliptic 
differential operator of order $2$.

Let us assume $E$ is a normal bundle and $\Sigma_0$ is minimal, 
\textit{i.e.} $\Hbf(0)=0$. We can 
compute the differential of $\Hbf$ with respect to $U$ at $0$, this gives
\[
D\Hbf(0)(V)=\Lbf(V)=\Delta^\perp V + (R(e_i,V)e_i)^\perp+ 
(V,B(e_i,e_j))B(e_i,e_j)
\]
which is an elliptic self-adjoint operator on $\Gamma(E)$. So $\Lbf$ has 
a discrete spectrum $\lambda_0\le \lambda_1 \le \cdots$. Let us notice that 
$\Sigma_0$ is unstable if $\lambda_0<0$. So the main theorem of the section is 
the following

\begin{thm}\label{thm:ancientmcf}
Let $E\to\Sigma$ be as above. Assume that the first eigenvalue 
$\lambda_0$ of $\Lbf$ is negative. Then for any section $V$ in the first 
eigenspace, there is $U$ an ancient solution of \eqref{eq:bmcf} defined on 
$(-\infty, b)$ such that 
\[
\lim_{t\to-\infty}e^{\lambda_0 t}U_t=V
\]
\end{thm}

One can compare this result with \cite[Theorem 1.6 and Theorem 3.3]{ChMa} by 
Choi and Mantoulidis. The main difference is that here the ancient solution is 
parametrized by its asymptotic as $t\to-\infty$ while Choi and Mantoulidis 
parametrized it by its value at time $t=0$.

\subsection{The functional spaces}

In order to prove the above result we need to introduce some functional spaces.
Following Solonnikov \cite{Slk}, we recall the definition of the H\"older 
spaces.

Let $\Ome\subset \R^m$ be a smooth domain and $P=\Ome\times[a,b]$. Then for $u: 
P\to \R^N$ and $\beta\in(0,1)$, we define the H\"older semi-norms
\begin{gather*}
[u]_{\beta,P,x}=\sup_{(x,t)\neq(y,t)\in P}\frac{|u(x,t)-u(y,t)|}{|x-y|^\beta}\\
[u]_{\beta,P,t}=\sup_{(x,t)\neq(x,s)\in P}\frac{|u(x,t)-u(x,s)|}{|t-s|^\beta}
\end{gather*}
and the uniform norm
\[
\|u\|_{0,P}=\sup_{X\in P}|u(X)|
\]
For $\alpha\in(0,1)$ we define the combined H\"older semi-norms
\begin{gather*}
[u]_{2,\alpha,P,x}=[\pr_x^2u]_{\alpha,P,x}+[\pr_tu]_{\alpha,P,x}\\
[u]_{2,\alpha,P,t}=[\pr_xu]_{(1+\alpha)/2,P,t}+[\pr_x^2u]_{\alpha/2,P,t}
 +[\pr_t u]_{\alpha/2,P,t}
\end{gather*}
Finally we have the H\"older norms
\begin{gather*}
\|u\|_{0,\alpha,P}=\|u\|_{0,P}+[u]_{\alpha,P,x}+[u]_{\alpha/2,P,t}\\
\|u\|_{2,\alpha,P}=\sum_{i=0}^2\|\pr_x^iu\|_{0,P}+ \|\pr_t u\|_{0,P}+ 
[u]_{2,\alpha,P,x} 
+[u]_{2,\alpha,P,t}
\end{gather*}

When $u$ is defined on $\Ome$, $u$ does not depend on $t$ so all the terms 
corresponding to the $t$ variable disappear and we have the specific notations:
\begin{gather*}
|u|_{0,\alpha,\Ome}=\|u\|_{0,\Ome}+[u]_{\alpha,\Ome,x}\\
|u|_{2,\alpha,\Ome}= 
\sum_{i=0}^2\|\pr_x^iu\|_{0,\Ome}+  
[u]_{2,\alpha,\Ome,x} 
\end{gather*}
We then have the associated H\"older spaces $C^{0,\alpha}(P)$, 
$C^{2,\alpha}(P)$, $C^{0,\alpha}(\Ome)$, $C^{2,\alpha}(\Ome)$ made of 
applications $u$ such that the above norms are well defined and finite.

This H\"older spaces can be analogously defined on a closed Riemannian manifold 
$(\Sigma,g)$ 
and for sections of a vector bundle $E$ over $\Sigma$ where $E$ is 
equipped with a bundle metric $h$ and a metric connection $\nb$. The 
vector bundle $E$ can 
be extended as a vector bundle (still denoted by $E$) over $\Sigma\times \R$. 
So if $P=\Sigma\times[a,b]$ and $U:P\to E$ is a section, we can define the 
H\"older semi-norms
\begin{gather*}
[U]_{\beta,E_{[a,b]},x}=\sup_{\substack{(x,t)\neq (y,t)\in 
P\\d_g(x,y)<i_g}}\frac{|U(x,t)-P_{y,x}U(y,t)|}{|x-y|^\beta}\\
[U]_{\beta,E_{[a,b]},t}=\sup_{(x,t)\neq(x,s)\in 
P}\frac{|U(x,t)-U(x,s)|}{|t-s|^\beta}
\end{gather*}
where $i_g$ denotes the injectivity radius of $\Sigma$ and $P_{y,x}$ is the 
parallel transport operator from $y$ to $x$. Once this is defined we can 
construct the H\"older norms similarly to the Euclidean case. The uniform norm:
\[
\|U\|_{0,E_{[a,b]}}=\sup_{X\in P}|U(X)|
\]
For $\alpha\in(0,1)$ we define the combined H\"older semi-norms
\begin{gather*}
[U]_{2,\alpha,E_{[a,b]},x}=[\nb_x^2U]_{\alpha,E_{[a,b]},x} 
+[\pr_tU]_{\alpha,E_{[a,b]},x}\\
[U]_{2,\alpha,E_{[a,b]},t}=[\nb_xU]_{(1+\alpha)/2,E_{[a,b]},t} 
+[\nb_x^2U]_{\alpha/2,E_{[a,b]},t} +[\pr_t U]_{\alpha/2,E_{[a,b]},t}
\end{gather*}
Finally we have the H\"older norms
\begin{gather*}
\|U\|_{0,\alpha,E_{[a,b]}}=\|U\|_{0,E_{[a,b]}}+[U]_{\alpha,E_{[a,b]},x} 
+[U]_{\alpha/2,E_{[a,b]},t}\\
\|U\|_{2,\alpha,E_{[a,b]}}=\sum_{i=0}^2\|\nb_x^iU\|_{0,E_{[a,b]}}+ \|\pr_t 
U\|_{0,E_{[a,b]}}+ [U]_{2,\alpha,E_{[a,b]},x} +[U]_{2,\alpha,E_{[a,b]},t}
\end{gather*}

When $U$ is defined on $\Sigma$,  we have the specific notations:
\begin{gather*}
|U|_{0,\alpha,E}=\|U\|_{0,\Sigma}+[U]_{\alpha,E,x}\\
|U|_{2,\alpha,E}= 
\sum_{i=0}^2\|\pr_x^iU\|_{0,E}+  
[U]_{2,\alpha,E,x} 
\end{gather*}
We then have the associated H\"older spaces $C^{0,\alpha}(E_{[a,b]})$, 
$C^{2,\alpha}(E_{[a,b]})$, $C^{0,\alpha}(E)$, $C^{2,\alpha}(E)$. In 
the sequel we will also use the $L^2$ norms $\|\cdot\|_{L^2(E_{[a,b]})}$ and 
$|\cdot|_{L^2(E)}$. For  a section $U$ defined over $\Sigma\times \R$, we 
denote 
$U_t(\cdot)=U(\cdot,t)$.

\subsection{Linear operators}

If the fiber of $E$ has dimension $k$, sections of can locally be written has 
maps: $u=(u^a)_{1\le a\le k}:\Ome \to \R^k$. In the sequel, we consider 
families of linear 
differential operators of order $2$ acting on 
sections of $E$ which in coordinates takes the form
\begin{equation}\label{eq:operator}
(L_tu)^a=\sum_{|I|\le 2,b\le k}A_b^{aI}(x,t)\pr_Iu^b
\end{equation}
where $I$ denote a multi-index and $\pr_I$ is the partial derivative associated 
to $I$. $L_t$ will be elliptic in the following sense: there is a  
constant $\lambda>0$ such that for any $\xi=(\xi_1,\dots,\xi_n)$ and 
$v=(v_1,\dots,v_k)$ we have
\[
\sum_{i,j=1}^2A_b^{aij}\xi_i\xi_jv_bv_a\ge \lambda |\xi|^2|v|^2
\]

Moreover we say that $L_t$ has $C^\alpha$ coefficients if the functions 
$A_b^{aI}$ are in $C^{0,\alpha}$. We denote by $\Lambda$ the maximum of the 
$C^{0,\alpha}$ 
norms of these coefficients. 

An important result for us is the following Schauder estimate for solutions of 
parabolic systems associated to such operators $L_t$

\begin{thm}[{\cite[Theorem 4.11]{Slk}}]
Let $\Ome'\subset \Ome\subset \R^n$ be smooth bounded domains with 
$\barre{\Ome'}\subset \Ome$. Let $P=\Ome\times [0,T]$ and $P'=\Ome'\times 
[0,T]$. Let $L_t$ be elliptic differential operators of order $2$ as in 
\eqref{eq:operator} with $C^\alpha$ coefficients in $\barre \Ome$. Then 
there 
is a constant $C$ depending on $\Ome$, $\Ome'$, $\lambda$, $\Lambda$, $\alpha$ 
and $T$ 
such that for any $u\in C^{2,\alpha}(\barre\Omega,\R^k)$ and $f\in 
C^{0,\alpha}(\barre P,\R^k)$ satisfying $\pr_tu-L_tu=f$ we have
\[
\|u\|_{2,\alpha,P'}\le 
C(\|f\|_{0,\alpha,P}+|u_0|_{2,\alpha,\Ome}+\|u\|_{L^2(P)})
\]
where $u_0(\cdot)=u(\cdot,0)$.
\end{thm}

Similar estimates can also be found in Friedman's paper~\cite{Fri}.

Using finitely many local charts for a vector bundle $E\to \Sigma$ ($\Sigma$ is 
closed), we can obtain an equivalent version for operators acting on sections 
of $E$.

\begin{thm}[{\cite[Theorem 4.11]{Slk}}]\label{thm:schauder}
Let $E$ be a vector bundle over a closed manifold $\Sigma$. Let $L_t$ be 
elliptic differential operators of order $2$ as in 
\eqref{eq:operator} in any local charts with $C^\alpha$ coefficients. Then 
there is a 
constant $C$ such that, for any $U\in C^{2,\alpha}(E_{[0,T]})$ and $F\in 
C^{0,\alpha}(E_{[0,T]})$ satisfying $\pr_tU-L_tU=F$, we have
\[
\|U\|_{2,\alpha,E_{[0,T]}}\le 
C(\|F\|_{0,\alpha,E_{[0,T]}}+|U_0|_{2,\alpha,E}+\|U\|_{L^2(E_{[0,T]})})
\]
\end{thm}

A consequence is the following solution to the Cauchy problem

\begin{thm}\label{thm:cauchy}
Let $E$ be a vector bundle over a closed manifold $\Sigma$. Let $L_t$ be 
elliptic differential operators of order $2$ as in 
\eqref{eq:operator}  in any local charts with $C^\alpha$ coefficients. Then, 
for any $U_0\in C^{2,\alpha}(E)$ and $F\in C^{0,\alpha}(E_{[0,T]})$, there is a 
unique $U\in C^{2,\alpha}(E_{[0,T]})$ such that
\[
\begin{cases}
\pr_t U-L_tU(\cdot,t)=F(\cdot,t)\\
U(\cdot,0)=U_0
\end{cases}
\]
\end{thm}
For a proof see \cite[Theorem 2.4]{Hua}.

In Theorem~\ref{thm:schauder}, the constant $C$ depends on the length $T$ of 
the time interval: actually it is uniformly bounded as $T\to 0$ but not as 
$T\to \infty$. However the proof can be adapted in order to obtain the 
following result where the constant is time independent. This is important for 
 our following arguments.

\begin{prop}\label{prop:schauder}
Let $E$ be a vector bundle over a closed manifold $\Sigma$. Let $L$ be 
a time independent elliptic differential operator of order $2$ as in 
\eqref{eq:operator} in any local charts with $C^\alpha$ coefficients . Then 
there is a 
constant $C$ (independent of $T$) such that for any $U\in 
C^{2,\alpha}(E_{[0,T]})$ and $F\in 
C^{0,\alpha}(E_{[0,T]})$ satisfying $\pr_tU-LU=F$ we have
\[
\|U\|_{2,\alpha,E_{[0,T]}}\le 
C(\|F\|_{0,\alpha,E_{[0,T]}}+|U_0|_{2,\alpha,E}+\|U\|_{L^2(E_{[0,T]})})
\]
\end{prop}

See the proof in Appendix~\ref{app:schauder}

\subsection{The ancient flow}

In this section we prove Theorem~\ref{thm:ancientmcf}. So $E\to \Sigma$ is a 
vector bundle as in Theorem~\ref{thm:ancientmcf} and we use the notations 
introduced in the preceding sections. We start by giving a 
result that ensures the existence of solutions to \eqref{eq:bmcf}.

\begin{thm}\label{thm:solmcf}
Let $E\to \Sigma$ as above. There is $\delta_0$ such that for any 
$\delta<\delta_0$ there is $\eps>0$ such that, for any $W\in 
C^{2,\alpha}(E)$ with $|W|_{2,\alpha,E}\le\eps$, there is a unique solution
$U\in 
C^{2,\alpha}(E_{[0,1]})$ of
\[
\begin{cases}
\pr_t U=\Hbf(U)\\
U(\cdot,0)=W
\end{cases}
\]
with $\|U\|_{2,\alpha,E_{[0,1]}}<\delta$.
\end{thm}

\begin{proof}
Let us consider the map
\[
F:\begin{array}{ccc}
C^{2,\alpha}(E)\times C^{2,\alpha}(E_{[0,1]})&\to 
&C^{2,\alpha}(E)\times 
C^{0,\alpha}(E_{[0,1]}) \\
(W,U)&\mapsto&(U(\cdot,0)-W,\pr_tU-\Hbf(U))
\end{array}
\]
$F$ is a $C^1$ map and $F(0,0)=(0,0)$ since $\Sigma_0$ is minimal. If we 
compute the differential of $F$ with respect to $U$ at $(0,0)$ we have
\[
D_UF(0,0):
\begin{array}{ccc}
C^{2,\alpha}(E_{[0,1]})&\to &C^{2,\alpha}(E)\times 
C^{0,\alpha}(E_{[0,1]}) \\
Z&\mapsto &(Z(\cdot,0),\pr_tZ-\Lbf Z)
\end{array}
\]
So the invertibility of this differential is given by the solution to the 
Cauchy problem (Theorem~\ref{thm:cauchy}). Hence the implicit function theorem 
solves $F(W,U)=(0,0)$ for any $W$ with $|W|_{2,\alpha,E}$ small.
\end{proof}

The above theorem produces solutions to the bundle MCF \eqref{eq:bmcf}. Let 
$\eps(\delta)$ be given by Theorem~\ref{thm:solmcf} for $\delta<\delta_0$. 
Actually it allows you to extend a solution $U$ as long as 
$|U_t|_{2,\alpha,\Sigma}<\eps(\delta)$.
\begin{prop}\label{prop:prolong}
Let $\delta<\delta_0$. Let $U$ be a solution of the bundle MCF defined on 
$\Sigma\times[a,b]$ with 
$\|X\|_{2,\alpha,E_{[a,b]}}\le \delta$. Let $\bar t\in(b-1,b)$ and assume that 
$|U_{\bar t}|_{2,\alpha,\Sigma}\le \eps(\delta)$ then $U$ can be extended as a 
solution of the bundle MCF defined on $\Sigma\times[a,\bar t+1]$
\end{prop}

\begin{proof}
Let $Z$ be the solution of \eqref{eq:bmcf} defined on $\Sigma\times[\bar t,\bar 
t+1]$ with $Z(\cdot,\bar t)=U(\cdot,\bar t)$ given by Theorem~\ref{thm:solmcf}. 
It suffices to prove that $Z=U$ on $\Sigma\times[\bar t,b]$ to conclude. This 
uniqueness is given by the following remark: we have
\begin{align*}
\pr_t(Z-U)=\Hbf(Z_t)-\Hbf(U_t)&=\int_0^1\frac{d}{ds}\Hbf(sZ_t+(1-s)U_t)ds\\
&=\int_0^1D\Hbf(sZ_t+(1-s)U_t)(Z_t-U_t)ds\\
&=L_t(Z_t-U_t)
\end{align*}
where $L_t$ are elliptic linear differential operators of order $2$ acting on 
sections of $E$ with coefficient in $C^\alpha$. Then by the uniqueness part of 
Theorem~\ref{thm:cauchy} and since 
$(Z-U)_{\bar t}=0$ we have $Z-U=0$ on $\Sigma\times[\bar t,b]$.
\end{proof}

To prove Theorem~\ref{thm:ancientmcf}, we consider $V$ an eigen-section 
associated to $\lambda_0$. We chose $\delta>0$ as in 
Theorem~\ref{thm:solmcf}. Let $a_\delta$ be such that $e^{-\lambda_0 
a_\delta}|V|_{2,\alpha,E}=\eps=\eps(\delta)$. Then for any $a<a_\delta$, 
$e^{-\lambda_0 a}|V|_{2,\alpha,E}<\eps$ so we can consider the section 
$U^{(a)}$ solution to the problem
\[
\begin{cases}
\pr_t U=\Hbf(U)\\
U(\cdot,a)=e^{-\lambda_0 a}V
\end{cases}
\]
on $\Sigma\times[a,b]$ where $b$ is chosen the largest possible such that 
$\|U^{(a)}\|_{2,\alpha,E_{[a,b]}}\le\delta$, $|U^{(a)}_t|_{2,\alpha,\Sigma}\le 
\eps$ and $\|e^{-\lambda_0t}V\|_{2,\alpha,E_{[a,b]}}\le\delta$ . So the proof 
consists in estimating the 
norm of $U^{(a)}$ in order to control $b$ and prove that, as $a\to -\infty$,
$U^{(a)}$ 
converges to the desired solutions of \eqref{eq:bmcf}.

Let us introduce $Z^{(a)}=U^{(a)}-e^{-\lambda_0 t}V$ for $t\in[a,b]$. We 
have the following result.
\begin{lem}\label{lem:estim}
There is $\delta>0$ and $b_0\in \R$ such that for any $a<\min(a_\delta,b_0)$, 
$U^{(a)}$ is defined on $[a,b_0]$. Moreover for any $a\le b\le b_0$
\[
\|Z^{(a)}\|_{L^2(E_{[a,b]})}\le e^{-3\lambda_0 b/2}
\]
\end{lem}

\begin{proof}
Let us choose $\delta\in(0,\delta_0)$ as in Theorem~\ref{thm:solmcf} that will 
be fixed below. Let us write $\Hbf=\Lbf+G$ where $G$ satisfies 
$|G(U)|_{\alpha,E}\le C |U|_{2,\alpha,E}^2$ for any section $U$ of $E$ with 
$|U|_{2,\alpha,E}\le \delta_0$. Actually $G$ satisfies 
$\|G(U)\|_{\alpha,E_{[a,b]}}\le C \|U\|_{2,\alpha,E_{[a,b]}}^2$ for any section 
$U$ of $E_{[a,b]}$ with $\|U\|_{2,\alpha,E_{[a,b]}}\le \delta_0$ and $C$ 
independent of $a,b$. In the computation below, the constant $C$ will change 
line to line but independently of $a$.

Let $b_\delta$ be such that 
$\|e^{-\lambda_0t}V\|_{2,\alpha,E_{(-\infty,b_\delta]}}\le\delta$. 
Then for $a<\min(a_\delta,b_\delta)$, we consider the solution $U^{(a)}$ 
defined on $[a,b]$, 
then $Z^{(a)}$ satisfies 
\begin{equation}\label{eq:dtZ}
\pr_t Z^{(a)}=\pr_t 
U^{(a)}+\lambda_0e^{-\lambda_0t}V=\Lbf(Z^{(a)})+G(U^{(a)})=\Lbf(Z^{(a)})+G(Z^{(a)}
 +e^{-\lambda_0t}V)
\end{equation}
Since $\|Z^{(a)}\|_{2,\alpha,E_{[a,b]}}\le 2\delta$, for $c\in[a,b]$, 
Solonnikov's estimate (Proposition~\ref{prop:schauder}) gives:
\begin{align*}
\|Z^{(a)}\|_{2,\alpha,E_{[a,c]}}&\le 
C(\|Z^{(a)}\|_{L^2(E_{[a,c]})}+\|G(Z^{(a)} 
+e^{-\lambda_0t}V)\|_{\alpha,E_{[a,c]}})\\
&\le C(\|Z^{(a)}\|_{L^2(E_{[a,c]})} +C(\|Z^{(a)}\|_{2,\alpha,E_{[a,c]}}^2 
+e^{-2\lambda_0 c}))\\
&\le 
C(\|Z^{(a)}\|_{L^2(E_{[a,c]})} +C(\delta\|Z^{(a)}\|_{2,\alpha,E_{[a,c]}} 
+e^{-2\lambda_0 c}))
\end{align*}
So we can choose and fix $\delta$ small enough such that $C\delta<1$ to obtain:
\begin{equation}\label{eq:schauder}
\|Z^{(a)}\|_{2,\alpha,E_{[a,c]}}\le C(\|Z^{(a)}\|_{L^2(E_{[a,c]})} 
+e^{-2\lambda_0c})
\end{equation}

So if $\|Z^{(a)}\|_{L^2(E_{[a,c]})}\le e^{-3\lambda_0c/2}$, we obtain 
$\|Z^{(a)}\|_{2,\alpha,E_{[a,c]}}\le Ce^{-3\lambda_0 c/2}$ and 
$\|U^{(a)}\|_{2,\alpha,E_{[a,c]}}\le Ce^{-\lambda_0 c}\le \min(\delta,\eps)$ if 
$c$ is less 
than some $\bar c$ (we restrict the definition of $U^{(a)}$ to $(-\infty,\bar 
c]$). So as long as the estimate $\|Z^{(a)}\|_{L^2(E_{[a,c]})}\le 
e^{-3\lambda_0 c/2}$ is true the solution $U^{(a)}$ is well defined. Let us now 
prove the estimate.

Since $Z^{(a)}(\cdot,a)=0$ the estimate is true at $c=a$. So let $c_1$ denote 
the 
first 
time where the estimate is not true. Because of \eqref{eq:dtZ}, we have the 
expression 
\[
Z^{(a)}_t=\int_a^t e^{(t-s)\Lbf}G(\Za_s+e^{-\lambda_0s}V)ds
\]
Since $\lambda_0$ is the first eigenvalue of $\Lbf$ we have
\begin{align*}
|\Za_t|_{L^2(E)}&\le 
\int_a^te^{-\lambda_0(t-s)}|G(\Za_s+e^{-\lambda_0s}V)|_{L^2(E)}ds\\
&\le\int_a^te^{-\lambda_0(t-s)}C|G(\Za_s+e^{-\lambda_0s}V)|_{0,E}ds\\
&\le\int_a^te^{-\lambda_0(t-s)}C(|\Za_s|_{2,\alpha,E}^2+e^{-2\lambda_0s})ds\\
&\le\int_a^te^{-\lambda_0(t-s)}C(\|\Za\|_{L^2(E_{[a,s]})}^2+e^{-2\lambda_0s})ds\\
&\le\int_a^te^{-\lambda_0(t-s)}C(e^{-3\lambda_0s}+e^{-2\lambda_0s})ds\\
&\le Ce^{-\lambda_0t}\int_a^te^{-\lambda_0s}ds\le Ce^{-2\lambda_0t}
\end{align*}
Then $\|\Za\|_{L^2(E_{[a,c_1]})}\le Ce^{-2\lambda c_1}$. So we see that $c_1$ 
must satisfies $Ce^{-\lambda_0 c_1/2}\ge 1$; \textit{i.e.} $c_1$ is bounded 
below by some universal constant $c_0$. So the Lemma is proved with 
$b_0=\min(c_0,\bar c,b_\delta)$.
\end{proof}

Let $b_0$ be given by Lemma~\ref{lem:estim}. By \eqref{eq:schauder}, we have 
$\|\Za\|_{2,\alpha,E_{[a,b_0]}}\le Ce^{-3\lambda_0 b_0/2}$. So by Arzela-Ascoli 
theorem, there is $Z\in C^{2,\alpha}(E_{(-\infty,b_0]})$, such that $\Za$ 
subconverge in $C^2$ to $Z$. Moreover, $Z$ satisfies 
$\pr_tZ=\Lbf Z+G(Z+e^{-\lambda_0 t}V)$, \textit{i.e.} $U=Z+e^{-\lambda_0 t}V$ 
is 
a 
solution of \eqref{eq:bmcf}. Since $\|\Za\|_{2,\alpha,E_{[a,t]}}\le 
Ce^{-3\lambda_0 t/2}$ for $t\le b_0$, we have 
$\|Z\|_{2,\alpha,E_{(-\infty,t]}}\le Ce^{-3\lambda_0 t/2}$ and then 
$\lim_{t\to-\infty}e^{\lambda_0 t}X_t=V$ in $C^{2,\alpha}$.

\section{The rigidity result}\label{sec:rigidity}

In this section we prove a rigidity result concerning 
$\S_1^n=\{p\in\R^{n+1}\mid \|p\|=1\}$ endowed with the induced metric 
$g_{\S_1^n}$. 
For $0\le k \le n-1$, let us consider the map:
\[
\Psi:\begin{array}{ccc}
\S_1^k\times\R\times \S_1^{n-k-1}&\longrightarrow &\S_1^n\\
(p,s,q)&\longmapsto&((\cos s) p,(\sin s) q)
\end{array}
\]
We notice that $\Psi(\S_1^k\times [0,\frac\pi2]\times \S_1^{n-k-1})=\S_1^n$, 
$\Psi$ 
is injective on $\S_1^k\times (0,\frac\pi2)\times \S_1^{n-k-1}$, 
$\Psi(p,0,q)=(p,0)$ and $\Psi(p,\frac\pi2,q)=(0,q)$. So $\S_1^n$ can be seen as 
the 
joint of $\S_1^k$ and $\S_1^{n-k-1}$. Moreover $\Psi^*(g_{\S_1^n})= \cos^2s 
g_{\S_1^k}+ds^2+\sin^2s g_{\S_1^{n-k-1}}$. The curves $s\mapsto \Psi(p,s,q)$ 
are geodesics of $\S_1^n$.

For $k=2$, we see that $\Psi(\S_1^2,0,q)$ is an immersed minimal sphere in 
$\S_1^n$ which is isometric to $\S_1^2$. Actually it is a totally geodesic 
equatorial $2$-sphere in $\S_1^n$. As a minimal surface its index is $n-2$. Our 
rigidity result looks at such an immersed sphere in a Riemannian manifold.

\begin{thm}\label{th:rigidity}
Let $M$ be a Riemannian $n\ge 3$-manifold whose sectional curvature is bounded 
above by $1$. Let us assume that $M$ contains an immersed minimal $2$-sphere of 
area $4\pi$ which is 
\begin{itemize}
\item either of index at least $n-2$
\item or unstable in any parallel directions.
\end{itemize}
Then the universal cover of $M$ is isometric to the sphere $\S_1^n$.
\end{thm}

\begin{proof}
Let $S$  be an immersed $2$-sphere in $M$, using Gauss and Gauss-Bonnet 
formulas 
we have
\begin{align*}
4\pi=\int_S K_S&=\int_S K_{TS}+(B(e_1,e_1),B(e_2,e_2))-\|B(e_1,e_2)\|^2\\
&=\int_S K_{TS}+\|\vec H\|^2-\|\inter B(e_1,e_1)\|^2-\|B(e_1,e_2)\|^2\\
&\le \boA(S)+\int_S\|\vec H\|^2-\frac12\|\inter B\|^2
\end{align*}
where $(e_1,e_2)$ is an orthonormal basis of $TS$, $K_S$ denotes the sectional 
curvature of $S$, $K_{TS}$ the sectional curvature of $M$ on the $2$-plane $TS$ 
and $\inter B$ is the 
traceless part of $B$. As a consequence
\[
F(S)=\boA(S)+\int_S\|\vec H\|^2-4\pi \ge \int_S \frac12\|\inter 
B\|^2\ge 0
\]
Hence $F(S)=0$ implies that $S$ is totally umbilic, $K_{TS}=1$ and 
$K_S=1+\|\vec 
H\|^2$. 

The immersed $2$-sphere in $M$ lifts to its universal cover with the same 
instability property. So we assume that $M$ is simply connected and $X:\S^2\to 
M$ 
an immersed minimal $2$-sphere as in the statement of the theorem. Let us 
notice 
that since 
$\Sigma=X(\S^2)$ is a minimal surface of area $4\pi$, 
$F(\Sigma)=0$ and $\Sigma$ is totally geodesic and $K_\Sigma=1$ so $\Sigma$ is 
isometric to $\S_1^2$: we can choose $X$ such that $X$ is 
an isometry between $\S_1^2$ and $\Sigma=X(\S^2)$.
Let us denote by $NX$ the normal vector bundle $\{(p,v)\in X^*TM\mid v\in 
T_pX^\perp\}$ and consider the map  
\[
\Phi : \begin{array}{ccc}
N X &\to& M\\
(p,v)&\mapsto& \exp_{X(p)}(v)
\end{array}
\]
We want to study the pull-back metric $h=\Phi^*g$ on $NX$ in order to 
control when $\Phi$ is an immersion.

The first step of the proof is
\begin{lem}\label{lem:eigen}
The normal bundle $NX$ is parallelizable. Moreover for any $p\in\S_1^2$ and 
unit 
vectors $e\in T_pX$ and $v\in T_pX^\perp$, $(R(e,v)e,v)=1$.
\end{lem}

\begin{proof}[Proof of Lemma~\ref{lem:eigen}]
By hypothesis $\Sigma$ is unstable so the Jacobi operator $L$ has a negative 
eigenvalue $\lambda$. Let us prove that $\lambda=-2$ and the associated 
eigen-section $V$ is a parallel section of $NX$.

We have $LV=-\lambda V$. For small $t$, we consider the immersed 
sphere $\Sigma_t=\{\Phi(p,tV(p)); p\in\S^2\}$. We then have $F(\Sigma_t)\ge 
0$ for any $t$. $F(\Sigma_0)=0$, so the first derivative of $t\mapsto 
F(\Sigma_t)$ must vanish at $t=0$: it is confirmed by the computation
\[
\frac{d}{dt}F(\Sigma_t)_{|t=0}=\int_\Sigma(-2\vec H,V)+\int_\Sigma\|\vec 
H\|^2(-2\vec H,V)+\int_\Sigma(LV,\vec H)=0
\]
Now the second derivative has to be non-negative and we have
\begin{align*}
\frac{d^2}{dt^2}F(\Sigma_t)_{|t=0}&=\int_\Sigma-(LV,V)+\int_\Sigma\frac12(LV,LV)\\
&=\int_\Sigma(\lambda+\frac12\lambda^2)\|V\|^2
\end{align*}
So $\lambda^2+2\lambda\ge 0$: $\lambda\le -2$. Since $\Sigma$ is totally 
geodesic, we also have
\begin{align*}
\lambda\int_\Sigma\|V\|^2&=\int_\Sigma(-LV,V)\\
&=\int_\Sigma\|\nabla^\perp V\|^2-(R(e_i,V)e_i,V)- (V,B(e_i,e_j))^2\\
&=\int_\Sigma\|\nabla^\perp V\|^2-(R(e_i,V)e_i,V)\ge \int_\Sigma -2\|V\|^2
\end{align*}
So $\lambda\ge -2$. This gives $\lambda=-2$. 

The above computation shows also that $V$ is a parallel normal vector 
field to $\Sigma$ and $(R(e,V)e,V)=1$ for any vector $e\in T\Sigma$.

Let $V_1,\dots,V_d$ be a basis of the eigenspace associated to the 
eigen-value $-2$. Let $\boB$ be the sub-bundle of $NX$ generated by 
$V_1,\dots,V_d$: $\boB=\{(p,v)\in NX \mid v\in \Span (V_1(p),\dots, 
V_d(p))\}$. $\boB$ is parallelizable and, on $\boB$, the stability operator is 
$L=-\Delta^\perp-2$. So the index of $L$ restricted to $\boB$ is precisely $d$. 
If $d<n-2$, both hypotheses on $\Sigma$ implies that the restriction of $L$ 
to $\boB^\perp$ must have a negative eigenvalue. Thus there is an eigensection 
of eigenvalue $-2$ in $\boB^\perp$ contradicting the definition of $\boB$. So 
$d=n-2$ and $\boB=NX$ which ends the proof.
\end{proof}

The sequel of the proof is a generalization of the above argument. 

Let us fix $V$ an eigen-section of $L$
associated to the eigenvalue $-2$. By Theorem~\ref{thm:ancientmcf}, let 
$(\Sigma_t)_{t\in(-\infty,b)}$ be the ancient solution of the mean curvature 
flow flowing out of $\Sigma$ in the direction $V$. $(-\infty,b)$ is a maximal 
time interval of existence.

We look at the evolution of $F(\Sigma_t)$. We know that $\lim_{t\to-\infty} 
F(\Sigma_t)=0$. Computing its derivative, we obtain
\begin{align*}
\frac{d}{dt} F(\Sigma_t)&=\int_{\Sigma_t}-4\|\vec H\|^2+\int_{\Sigma_t}-4\|\vec 
H\|^4+\int_{\Sigma_t}-2\|\nabla^\perp\vec H\|^2\\
&\qquad\qquad2(R(e_i,\vec H)e_i,\vec H)+2(\vec 
H,B(e_i,e_j))^2\\
&\le \int_{\Sigma_t} 2[(R(e_i,\vec H)e_i,\vec H)-2\|\vec H\|^2)] +
\int_{\Sigma_t} 2(\vec H,\inter B(e_i,e_j))^2\\
&\le \int_{\Sigma_t}2\|\vec H\|^2 \|\inter B\|^2\le 4\sup_{\Sigma_t} \|\vec 
H\|^2 F(\Sigma_t)
\end{align*}

By construction of $\Sigma_t$, we know that close to $-\infty$, 
$\sup_{\Sigma_t} 
\|\vec H\|^2\le Ce^{4t}$. Using this and $\lim_{t\to-\infty} F(\Sigma_t)=0$, a 
Gronwall type argument gives $F(\Sigma_t)\le 0$ for any $t$ and then $F(t)=0$ 
for any $t$. So 
$\inter B=0$ on $\Sigma_t$, $K_{T\Sigma_t}=1$.

This also implies that the derivative of $F(\Sigma_t)$ vanishes: $\vec H$ is a 
parallel normal vectorfield along $\Sigma_t$ and $(R(e,\vec H)e,\vec H)=\|\vec 
H\|^2$ for any unit vector $e\in T\Sigma_t$. Since $\vec H$ is parallel, we 
have $\|H\|$ is constant along $\Sigma_t$ (notice that $\|\vec H\|\neq 0$ by 
construction). So we 
can write $\vec H=H_t\nu$ ($H_t=\|\vec H\|(t)$) where $\nu$ is a unit normal 
vector 
field to $\Sigma_t$. Moreover $\nu$ is parallel along $\Sigma_t$.

Let us define a new time parameter $s=s(t)=\int_{-\infty}^tH_udu$, so that 
$\frac{ds}{dt}=H_t$. Hence the derivative of $\Sigma_s$ with respect to $s$ is 
given by $\nu$.

If $q\in \Sigma_s$, the map $(a,b)\mapsto(R(a,b)a,b)$ defined for unit vectors 
$a$, $b\in T_qM$ is bounded above by $1$ and reaches its maximum at 
$(a,b)=(f,\nu)$ where $f\in T_q\Sigma_s$. So we have $(R(f,\nu)f,v)=0$ for any 
$v\in 
\nu^\perp$ and $(R(f,\nu)v,\nu)=0$ for any $v\in f^\perp$.

We can compute $\frac{\barre D}{ds}2\vec H_s$ in two ways:

\begin{align*}
&\qquad\frac{\barre D}{ds}2\vec H_s=2\frac{d}{ds}H_s \nu+2H_s\frac{\barre 
D}{ds}\nu 
\qquad \text{and}\\
\frac{\barre D}{ds}2\vec H_s&=\Delta^\perp 
\nu+(R(e_i,\nu)e_i)^\perp+(\nu,B(e_i,e_j))B(e_i,e_j)-(\vec 
H,\nabla_{e_i}\nu)e_i\\
&=2\nu+2H_s^2\nu
\end{align*}
Hence $\frac{\barre D}{ds}\nu=0$; the evolution follows geodesics 
and $\Sigma_s=\{\Phi(p,sV(p));p\in\Sigma\}$. Besides $\frac{d}{ds}H_s=1+H_s^2$, 
so  $H_s=\tan s$. 

Let $\gamma_p$ be the geodesic $s\mapsto \Phi(p,sV(p))$. We are going to study 
some Jacobi fields along $\gamma_p$. Let $f_1,\dots,f_{n-1}$ be parallel 
orthonormal fields along $\gamma_p$ such that $f_1,\dots,f_{n-1},\nu$ is 
orthonormal and, at $s=0$, $f_1, f_2$ is a basis of $T_pX$. For $i\in\{1,2\}$ 
we define $\pr_i$ the Jacobi field along $\gamma_p$ such that $\pr_i(0)=f_i$ 
and $\frac{\barre D}{ds}{\pr_i}(0)=0$. Actually if $e_i\in T_p\S^2$ is 
such that $X_*(e_i)=f_i$, we have $\pr_i=D_{e_i}(\Phi(\cdot,sV(\cdot))$ so 
$\pr_i(s)$ is tangent to $\Sigma_s$ as long as $\Sigma_s$ is well defined. 
$\pr_i$ is a Jacobi field 
so because of the above computations of the Riemann tensor
\[
0=\frac{\barre D^2}{ds^2}\pr_i+R(\nu,\pr_i)\nu=\frac{\barre 
D^2}{ds^2}\pr_i+\pr_i
\]
Decomposing this equation in $(f_1, \cdots,f_{n-1})$,
 we obtain that $\pr_i=\cos s f_i$. 
We have $(\pr_1,\pr_2)$ is an orthogonal basis of $T\Sigma_s$: $\Sigma_s$ is an 
immersion for $s\in[0,\pi/2)$. 

As a consequence the orthogonal of $T\Sigma_s$ is generated by 
$(\nu,f_3,\cdots, 
f_{n-1})$. Let $\pr_j$ ($j\ge 3$) be the Jacobi fields along 
$\gamma_p$ with $\pr_j(0)=0$ and $\frac{\barre D}{ds}\pr_j(0)=f_j$. We have 
$(R(\nu,\pr_j)\nu,f_i)=(\pr_j,f_i)$ for $i\in\{1,2\}$ and $s\in[0,\frac\pi2]$. 
So 
$(\pr_j,f_i)$ is solution of the ODE $y''+y=0$ with vanishing initial value and 
derivative. Thus $(\pr_j,f_i)=0$ and $\pr_j$ belongs to 
$\Span(f_3,\cdots,f_{n-1})$. Moreover Rauch comparison theorem implies that 
$\pr_j$, $j\in\{3,\dots,n-1\}$, are non vanishing on $(0,\pi)$.

We know that $NX$ is parallelizable so we can fix an isometric 
parametrization by $NX\simeq\S_1^2\times\R^{n-2}$ and we have a map 
\[
\Phi:\S_1^2\times\R^{n-2}\to M
\]
Using a polar decomposition of $\R^{n-2}$ as $\R_+\times \S^{n-3}$ and 
coordinates $(p,s,q)\in \S^2\times\R_+\times \S^{n-3}$, this gives a map $\Psi: 
\S^2\times\R_+\times \S^{n-3}\to M$ defined by $\Psi(p,s,q)=\Phi(p,sq)$. The 
above study of the Jacobi fields along the geodesic gives that the lift of the 
metric is given by
\[
\Psi^* g=\cos^2s g_{\S_1^2}+ds^2+g_{p,s}
\]
for $s\in[0,\pi/2]$ and $g_{p,s}$ is a smooth family of metrics on $\S^{n-3}$ 
depending on $(p,s)\in \S^2\times(0,\pi/2)$. Let us notice that 
$g_{p,s}=s^2g_{\S_1^{n-3}}+o_0(s^2)$ and $g_{\frac\pi2,p}$ is a well defined 
metric on $\S^{n-3}$.

As a consequence $\Psi(p,\frac\pi2,q)$ is a point $Q\in M$ that does not depend 
on $p$. If $\bar p$ is fixed $\Sigma'=\Psi(\bar p,\frac\pi2,\S^{n-3})$ is then 
an immersed submanifold of $M$. Let us study the geometry near $\Sigma'$. The 
geodesics $s\mapsto\Psi(p,s,q)$ arrive orthogonally to $\Sigma'$ when 
$s=\frac\pi2$. Let us fix $\bar q\in \S^{n-3}$ and define a map
\[
G:
\begin{array}{ccc}
\S^2&\to &U_{\bar q}X'^\perp\\
p&\mapsto&\frac d{ds}\Psi(p,s,\bar q)_{|s=\frac\pi2}
\end{array}
\]
where $U_{\bar q}X'^\perp$ is the unit sphere in the normal bundle $TX'^\perp$ 
at 
$\bar q$.

For $r\in[0,\pi/2]$ let us define $F_r:\S^2\to M, p\mapsto \exp_{\barre 
Q}(-rG(p))$. We have $F_r(p)=\Psi(\bar p,\frac\pi2-r,\bar q)$ so ${F_r}^*g=
\sin^2r g_{\S_1^2}$. So $G^*g_{\barre Q}=\lim_{r\to 
0}\frac1{r^2}{F_r}^*g=g_{\S_1^2}$. So $G$ is a linear isometry between $\S^2$ 
and $U_{\bar q}X'^\perp$. As a consequence $G(-p)=-G(p)$. Thus $\Psi( 
p,s+\frac\pi2,\bar q)=\exp_{\barre Q}(sG(p))=\exp_{\barre 
Q}(-sG(-p))=\Psi(-p,\frac\pi2-s,\bar q)$ for $s\in[0,\frac\pi2]$. So 
$\Psi(p,\pi,q)=\Psi(-p,0,q)=X(-p)$. This implies that $\pi$ is a conjugate time 
for the Jacobi fields $\partial_j$ ($3\le j\le n-1$). Then by Rauch comparison 
theorem, $|\pr_j(s)|=|\sin s|$ and $(R(\pr_j,\nu)\pr_j,\nu)=\sin^2 s$. Thus 
$R(f_j,\nu)f_j,\nu)=1$ for any $j$. So $\Psi^*g=\cos^2sg_{\S_1^2}+ds^2+\sin^2s 
g_{\S_1^{n-3}}$. As a consequence, $\Psi$ generates a local isometry $\Psi'$ 
from the joint of $\S^2$ with $\S^{n-3}$ endowed with the metric 
$\cos^2sg_{\S_1^2}+ds^2+\sin^2s 
g_{\S_1^{n-3}}$ to $M$. So $\Psi'$ is a local isometry and thus a covering 
map from $\S_1^n$ to $M$. Since $M$ is simply connected $\Psi'$ is a global 
isometry.
\end{proof}

The above result has a corollary concerning manifold containing an "equator". 
This an infinitesimal version of \cite[Corollary 5.11]{AnHo}.

\begin{cor}
Let $M$ be a Riemannian $n\ge 3$-manifold whose sectional curvature is bounded 
above by $1$. Let us assume that there is a totally geodesic isometric 
immersion $f : \S_1^{n-1}\to M$. Moreover, assume that $f$ is unstable as 
a minimal hypersurface. Then the universal cover of $M$ is isometric to 
$\S_1^n$.
\end{cor}

\begin{proof}
Let $G(2,n-1)$ be the set of totally geodesic $2$ spheres in $\S_1^{n-1}$: the 
intersections of $\S_1^{n-1}$ with any $3$-plane in $\R^n$. For any $S$ in 
$G(2,n-1)$, $f(S)$ is totally geodesic in $M$. Moreover $f(S)$ has index at 
least $n-3$. So it is enough to prove that one of these $S$ has index at least 
$n-2$ to conclude by Theorem~\ref{th:rigidity}. For any $S\in G(2,n-1)$, let 
$Q_S$ be the quadratic form 
associated to the Jacobi operator on $f(S)$.

Let $\nu$ be the unit normal to $f(\S_1^{n-1})$. For $S\in G(2,n-1)$, 
$\Span(\nu)$ is a parallel normal bundle along $f(S)$. Let $\lambda_0<0$ be the 
first eigenvalue of the Jacobi operator on $f(\S_1^{n-1})$ and $u$ a first 
eigenfunction. On $\S_1^{n-1}$, we can define a quadratic form $q=du\otimes du 
- (R(\nu,\cdot)\nu,\cdot)$. Then the quadratic form $Q_{\S_1^{n-1}}$ associated 
to 
the Jacobi operator on $f(\S_1^{n-1})$ satisfies to 
\[
Q_{\S_1^{n-1}}(u,u)=\int_{\S_1^{n-1}}\tr_{T\S_1^{n-1}}q
=\lambda_0\int_{S_1^{n-1}}u^2<0
\]

Now there is a dimensional constant $c_{n}$ such that
\begin{align*}
c_n\int_{\S_1^{n-1}}\tr_{T\S_1^{n-1}}q&=\int_{G(2,n-1)}dS\int_S\tr_{TS}q\\
&=\int_{G(2,n-1)}Q_S(u\nu,u\nu) dS
\end{align*}
where $\tr_P$ denotes the trace operator on the subspace $P$ and $dS$ is the 
Haar measure on $G(2,n-1)$ coming from the Haar measure on the Grassmannian of 
$3$-planes in $\R^{n+1}$.
Thus  there is $S$ such that $Q_S(u\nu,u\nu)<0$. So the restriction to 
$\Span(\nu)$ of the stability operator for $S$ is unstable: $S$ has index at 
least $n-2$.
\end{proof}

\section{Rigidity of Simon-Smith width}\label{sec:minmax}

Here we state Smith theorem concerning the existence of a minimal 
sphere in any Riemannian $3$-sphere $(\S^3,g)$ and our second rigidity result.

We start with the standard sweep-out of the sphere $\S^3$ given by 
horizontal spheres
\[
S_t=\{(x_1,x_2,x_3,x_4)\in \S^3\mid x_4=t\} \text{ for } t\in[-1,1]
\]
If $F_t:\S^3\to \S^3$ ($t\in[-1,1]$) is a smooth family of diffeomorphisms 
isotopic to the identity map, we can defined a general sweep-out of $\S^3$ by 
spheres as the family $\sigma_t=F_t(S_t)$ ($t\in[-1,1]$). Let $\Lambda$ denote 
the set of all general sweep-outs of $\S^3$ by spheres. The Simon-Smith width 
of a Riemannian $3$-sphere $(\S^3,g)$ is then defined by
\[
W(\S^3,g)= \inf_{\{\sigma_t\}\in \Lambda}(\max_{t\in[-1,1]} \boA(\sigma_t))
\]
For $\sigma\in \Lambda$, we denote $L(\sigma)=\max_{t\in[-1,1]}\boA(\sigma_t)$. 
So a 
minimizing sequence for the Simon-Smith width $W$ is a sequence 
$(\sigma^n)$ in $\Lambda$ such that $\lim L(\sigma^n)=W$. For such a minimizing 
sequence, one can consider a sequence $(t_n)\in[-1,1]$ such that $\lim 
\boA(\sigma_{t_n}^n)=W$. Such a sequence is called a min-max sequence.
The main result in \cite{Smi} is

\begin{thm}[Smith]\label{th:simonsmith}
There is a min-max sequence  that converges in the sense 
of varifolds to a disjoint union of embedded minimal spheres (possibly with 
multiplicity) whose area is $W(\S^3,g)$.
\end{thm}

Moreover the quantity $W(S^3,g)$ is positive so the collection of minimal 
$2$-spheres is not empty.

If the sectional curvature of $(\S^3,g)$ is bounded above by $1$ the area of 
each sphere in the collection is at least $4\pi$ so $W(\S^3,g)\ge 4\pi$. We 
notice that for the Euclidean sphere $\S_1^3$ we have $W(\S_1^3)=4\pi$. The 
main result of this section is a rigidity result for the equality case.

\begin{thm}\label{th:rigidityminmax}
Let $(\S^3,g)$ be a Riemannian $3$-sphere whose sectional curvature is bounded 
above by $1$. Then $W(\S^3,g)\ge 4\pi$ and, if $W(\S^3,g)=4\pi$, then 
$(\S^3,g)$ is isometric to $\S_1^3$.
\end{thm}

In order to give the proof of this result we need to understand the index 
property of the minimal spheres that realize $W(\S^3,g)$. In the sequel we 
introduce de notion of almost stable minimal hypersurfaces. We will use concept 
from geometric measure theory, for the notations we refer to \cite{CoDeL} and 
\cite{MaNe6}.

\subsection{The Simon-Smith min-max surface}

In this section we recall some aspect of the min-max construction by Smith in 
\cite{Smi}. Actually we mainly refer to the statements contained in the paper 
by Colding and De Lellis \cite{CoDeL}.

If the varifold $V$ is a varifold limit of some min-max-sequence we say that 
$V$ is a 
min-max varifold. So the proof of Theorem~\ref{th:simonsmith} consists in 
finding a min-max varifold which is smooth and has the right topological type.

The first step in this proof is that there is a minimizing sequence 
$(\sigma^n)$ such that any min-max varifold coming from $(\sigma^n)$ is 
stationary (see \cite[Proposition~4.1]{CoDeL}).
The second step toward regularity use the notion of almost-minimizing surface.

\begin{defn}
Given $\eps>0$, an open set $U\in \S^3$ and a surface $\Sigma$. $\Sigma$ is 
$\eps$-almost minimizing in $U$ if there is no isotopy $\psi$ supported in $U$ 
such 
that
\begin{gather*}
\boA(\psi_t(\Sigma))\le \boA(\Sigma)+\eps/8 \text{ for all }t\\
\boA(\psi_1(\Sigma))\le \boA(\Sigma)-\eps
\end{gather*}
A  sequence $(\Sigma_n)$ is said to be almost minimizing in $U$ if $\Sigma_n$ 
is $\eps_n$-almost minimizing in $U$ for some sequence $\eps_n\to 0$
\end{defn}

For $p$ and $r>0$, we define the set of annuli centered at $p$ of outer radius 
less that $r$ $\boA\boN_r(p)=\{B_t(p)\setminus \barre B_s(p), 0<s<t<r\}$. Then 
one can select a positive function on $\S^3$ and a min-max sequence 
$(\sigma_{t_n}^n)$ which is almost minimizing in any small annuli: 
\textit{i.e.} in any $A\in \boA\boN_{r(p)}(p)$ 
for all $p\in \S^3$ (see \cite[Proposition~5.1]{CoDeL}). Then the author proves 
that if $\sigma_{t_n}^n\to V^*$ as varifold, $V^*$ has the 
expected properties. Concerning smoothness, the author introduce the notion of 
replacement.

\begin{defn}
Let $V$ be a stationary varifold and $U$ is an open subset of $\S^3$. A 
stationary varifold 
$V'$ is said to be a replacement of $V$ in $U$ if
\begin{itemize}
\item $V'=V$ on $G(M\setminus \barre U)$ and $\|V'\|(M)=\|V\|(M)$ 
\item $V'\llcorner U$ is supported by a stable minimal surface $\Sigma$ with 
$\barre \Sigma\setminus \Sigma\subset \partial U$.
\end{itemize}
\end{defn}

The proof of the smoothness $V^*$ consist basically in proving that $V^*$ has 
replacements in any small annuli and that they coincide with $V^*$ (see 
\cite[Theorem~7.1]{CoDeL}). 

Let $(\Sigma_n)=(\sigma_{t_n}^n)$ we know that $\Sigma_n\to V^*$ as varifold. 
Moreover, viewing $\Sigma_n$ as a flat cycle, we can assume that $\Sigma_n \to 
T\in \boZ_2(M,\Z_2)$ as current. So 
one can ask the relation between $T$ and $V^*$. Since $V^*$ is made of smooth 
surface with integer multiplicities, there is a corresponding element 
$[V^*]\in\boZ_2(M,\Z_2)$ by reducing each multiplicities mod $2$. So one can 
suspect that $T=[V^*]$. This is confirmed by the result below (we are inspired 
by a similar result for the Almgrenn-Pitts min-max theory by Marques ans Neves 
\cite[Proposition 4.10]{MaNe5}). Actually the 
support of $T$ has to be contained in $\spt V^*$. So if $\{S_i\}_{1\le i\le N}$ 
is the collection of connected 
components of $\spt V^*$, the Constancy Theorem implies that $T=n_1S_1+\cdots 
+n_NS_N$ for some $n_i\in\{0,1\}$.

\begin{prop}\label{prop:currentlimit}
Let $(\Sigma_n)$ be a min-max sequence which is almost-minimizing in any small 
annuli and such that $\Sigma_n\to V$ as varifold and
\[
V=m_1 S_1+\cdots +m_N S_N
\]
where each $S_i$ is a smooth embedded minimal surface. Let us assume that 
$\Sigma_n\to T\in\boZ_2(M,\Z_2)$ as current and 
\[
T=n_1 S_1+\cdots +n_N S_N
\]
with $n_i\in\{0,1\}$. Then $m_i=n_i\mod 2$.
\end{prop}

One tool in the proof of the smoothness of $V^*$ is the following result of 
Meeks, Simon and Yau.
\begin{thm}[Meeks-Simon-Yau \cite{MeSiYa}]\label{th:msy}
Let $\Sigma$ be a surface in $M$ and $U$ an open subset of $M$.
Let $(\Sigma_k)$ be a minimizing sequence for the Problem $(\Sigma,\Iso(U))$, 
\textit{i.e.} $\Sigma_k=\psi(\Sigma)$ for some $\psi\in\Iso(U)$ and
\[
\boA(\Sigma_k)\to \inf_{\psi\in \Iso(U)}\boA(\psi(\Sigma)),
\] 
which converges to a varifold $V$. Then $V\llcorner U$ is an integer 
rectifiable varifold whose support is a stable minimal surface $\Gamma$ with 
$\barre \Gamma\setminus \Gamma\subset \partial U$. Moreover as current 
$\Sigma_k\llcorner U\to [V\llcorner U]$ in $\mathbf{I}_2(U,\Z_2)$.
\end{thm}

The proof of the regularity of $V\llcorner U$ is local so is contained in 
\cite{MeSiYa}. As above this regularity implies the existence of flat chain mod 
$2$, $[V\llcorner U]\in \mathbf{I}_2(U,\Z_2)$. The proof of the convergence as 
current is a byproduct of the regularity proof as noticed by Almgren and Simon 
in \cite[Remark 5.20]{AlSi}.

As mentioned above the smoothness proof uses the fact that replacement of $V^*$ 
coincide with $V^*$. This is a natural property of stationary varifold with 
smooth support.

\begin{lem}\label{lem:replacement}
Let $V$ be a integer stationary varifold whose support is a smooth minimal 
surface $S$ and $p\in S$. Then there is $\rho>0$ such the following is true: 
if 
$V'$ is a stationary varifold such that $V'\llcorner B_\rho(p)^c=V\llcorner 
B_\rho(p)^c$, then $V'=V$.
\end{lem}

\begin{proof}
Let $r>0$ be chosen such that $\Delta=S\cap B_r$ is connected, stable and with 
non-empty smooth boundary. Let $u$ be a non negative eigenfunction for the 
first eigenvalue of 
the Jacobi operator on $\Delta$ with $u=0$ on $\partial \Delta$. Let 
$\Delta_t=\{\exp_p(tu(p)\nu(p)), p\in \Delta\}$ where $\nu$ is the unit normal 
to $\Delta$. There is $\eps>0$ such that, for $|t|\le \eps$, $\Delta_t$ is a 
smooth surface with mean curvature pointing to $\Delta=\Delta_0$. Let $\rho$ 
such that $B_\rho(p)\subset \cup_{|t|<\eps}\Delta_t$.

Let $V'$ be as in the Lemma, we notice that $\spt V'\subset S\cup B_\rho(p)$. 
Using the surfaces $\Delta_t$ as barriers and the 
maximum principle~\cite{Whi7}, we have that $\spt V'\in S$. Then the Constancy 
theorem and $V'\llcorner B_\rho(p)^c=V\llcorner B_\rho(p)^c$ implies $V'=V$.
\end{proof}

We then have
\begin{proof}[Proof of Proposition~\ref{prop:currentlimit}]
Let $p\in S_1$ and $\rho$ given by Lemma~\ref{lem:replacement}. Let $A\subset 
B_\rho(p)$ be an annulus centered at $p$ such that $\Sigma_n$ is 
$\eps_n$-almost minimizing in $A$.

Let 
\[
\Iso_n(A)=\{\psi\in \Iso(A)\mid \boA(\psi_t(\Sigma_n))\le 
\boA(\Sigma_n)+\eps_n/8\}
\]
and $\psi^k\in\Iso_n(A)$ such that
\[
\boA(\psi_1^k(\Sigma_n))\to \inf_{\psi\in \Iso_n(A)}\boA(\psi_1(\Sigma_n))
\]

Let $q\in 
S_1 \cap A$ and $\eps>0$ such that $B_\eps(q)\subset A$. We first fix $n$ and 
$k$. Let 
$\phi^j\in\Iso(B_\eps(q))$ be a sequence such that 
\[
\boA(\phi_1^j(\Sigma_n^k))\to \inf_{\phi\in 
\Iso(B_\eps(q))}\boA(\phi_1(\Sigma_n^k))
\]

Let $W_n^k$ be the varifold limit of $\phi_1^j(\Sigma_n^k)$ and $R_n^k\in 
\boZ_2(M,\Z_2)$ be the limit of $\phi_1^j(\Sigma_n^k)$ for the flat 
convergence. By Theorem~\ref{th:msy}, 
$R_n^k\llcorner B_\eps(q)= [W_n^k\llcorner B_\eps(q)]\in 
\mathbf{I}_2(B_\eps(q),\Z_2)$. Besides we have 
$R_n^k\llcorner \barre B_\eps(q)^c = \Sigma_n^k\llcorner \barre B_\eps(q)^c$. 

Letting $k\to \infty$ we consider $W_n$ a varifold limit of $W_n^k$ and $R_n$ a 
limit of $R_n^k$ for the flat convergence. Since $\spt(W_n^k)\cap B_\eps(q)$ is 
a stable minimal surface, we have curvature estimates and the convergence of 
$\spt(W_n^k)\cap B_\eps(q)$ to $\spt(W_n)$ is locally smoothly graphical. So 
taking multiplicities of these graphical leaves into account,  varifold and 
flat convergences give $R_n\llcorner B_\eps(q)= 
[W_n\llcorner B_\eps(q)]$.

There is a subsequence $(\phi_1^{j(k)}(\Sigma_n^k))$ which 
converges to $W_n$ as varifold and $R_n$ as current. By \cite[Lemma 
7.6]{CoDeL} 
there is $\Phi^{j(k)}\in 
\Iso(B_\eps(q))$ such that  $\Phi_1^{j(k)}=\phi_1^{j(k)}$ and 
$\boA(\Phi_t^{j(k)}(\Sigma_n^k))\le 
\boA(\Sigma_n^k)+\eps_n/8$. This implies  that $\phi_1^{j(k)}(\Sigma_n^k)$ can 
be constructed from $\Sigma_n$ by an isotopy in  $\Iso_n(A)$.  Then 
$(\phi_1^{j(k)}(\Sigma_n^k))$ is a minimizing sequence in  $\Iso_n(A)$. So 
$\spt(W_n)\cap A$ is a stable minimal surface \cite[Lemma 7.4]{CoDeL}. Moreover 
a varifold limit $W$ of 
$W_n$ is a replacement for $V$ in $A$ \cite[Proposition 7.5]{CoDeL} and then 
$V=W$ because of Lemma~\ref{lem:replacement} and $A\subset B_\rho(p)$. Let $R$ 
be a current limit of $R_n$. 

As above since $\spt(W_n)\cap A$ is a stable minimal surface, the varifold and 
flat convergence give $R\llcorner A=[W\llcorner A]=[V\llcorner A]$. Actually 
$\spt R\subset \spt 
W=\spt V$ and $R\llcorner \barre A^c=T\llcorner \barre A^c$. So by the 
Constancy Theorem, 
$R=T$. So $T\llcorner A=[V\llcorner A]$, this implies that the multiplicity of 
$T$ on $S_1$ is equal to the one of $V$ mod $2$. Then $T=[V]$ by considering 
$p$ on other connected components.
\end{proof}

\subsection{Almost stable minimal hypersurfaces}

If $\Sigma$ is a stable minimal hypersurface in a Riemannian manifold $M$, then 
$\Sigma$ is a local minimum for the area functional (see \cite{Whi5}). In this 
section we introduce almost stable minimal hypersurfaces as hypersurfaces which 
have this local minimum property: one can give a sense to stability of a 
minimal hypersurface $\Sigma$ whose first Jacobi eigenvalue 
vanishes, \textit{i.e.} degenerate stable minimal hypersurface.

Actually we focus only on $2$-sided minimal hypersurfaces, the $1$-sided case 
can be considered similarly. So let $\Sigma$ be a $2$-sided minimal 
hypersurface and parameterize the $\eps$-neighborhood of $\Sigma$ by 
$\Phi:\Sigma\times(-\eps,\eps)\to M;(p,t)=\exp_p(t\nu(p))$. This defines a 
vectorfield $\partial_t$ on the neighborhood. If $u$ is a smooth 
function on $\Sigma$, we define $\Phi_u : \Sigma\to M$ by 
$\Phi_u(p)=\Phi(p,u(p))$ and $\Sigma_u=\Phi_u(\Sigma)$. If $u$ is small, 
$\Sigma_u$ is immersed and we define $\nu_u(p)$ the 
unit normal to $\Sigma_u$ at $\Phi_u(p)$ such that $(\nu_u,\pr_t)\ge 0$ and 
$H_u(p)\in \R$ such that the mean curvature vector of $\Sigma_u$ is 
$H_u(p)\nu_u(p)$  at $\Phi_u(p)$. Let $W_u$ be the jacobian of the map 
$\Phi_u$. Then the area of $\Sigma_u$ can be computed as 
$\boA(\Sigma_u)=A(u)=\int_\Sigma W_u$. For $u$ and $v$ two functions we then 
have 
\begin{equation}\label{eq:differentiate}
DA(u)(v)=\int_\Sigma H_u(\nu_u,\pr_t)vW_u=\int_\Sigma 
h_uv
\end{equation}
where $h_u=H_u(\nu_u,\pr_t)W_u$. The map $u\mapsto h_u$ is a smooth operator of 
order $2$. 

\begin{lem}\label{lem:foliat}
Let be a degenerate stable minimal $2$-sided hypersurface $\Sigma$ with 
 first Jacobi eigen-function $u_0$. There is $\eps>0$ and a 
 smooth 
map 
$v:(-\eps,\eps)\times\Sigma\to \R; (t,p)\mapsto v_t(p)$ with the following 
properties.
\begin{itemize}
\item $v_0=0$, $\pr_t {v_t}_{|t=0}=0$ and $\int_\Sigma u_0v_t=0$.
\item for each $t\in(-\eps,\eps)$, $h_{tu_0+v_t}$ is a multiple of $u_0$.
\end{itemize}
\end{lem}

\begin{proof}
Let $V$ be the $L^2$ orthogonal of $u_0$ and $\pi_V$ the orthogonal projection 
on $V$. Let $V^{k,\alpha}=V\cap C^{k,\alpha}(\Sigma)$. Let us define the map: 
$G : \R\times V^{2,\alpha}\to V^{0,\alpha}; (t,v)\mapsto 
\pi_V(h_{tu_0+v})$. We notice that $G(0,0)=0$ and $D_vG(0,0)(w)=\pi_V(Lw)$ 
where 
$L$ is the Jacobi operator. Since $L$ is invertible from $V^{2,\alpha}$ to 
$V^{0,\alpha}$, the implicit function theorem gives the map $v$.
\end{proof}

Once the family $\{v_t\}$ is constructed we then define a foliation of a 
neighborhood of $\Sigma$ by $S_t=\Sigma_{tu_0+v_t}$, this foliation is called 
the canonical foliation given by $\Sigma$. Let $a(t)=\boA(S_t)$.
We notice that by construction $a'(t)=0$ if and only if $S_t$ is a minimal 
surface. Moreover, for small $u$, $\Sigma_u$ is a minimal surface iff 
$\Sigma_u=S_t$ for some small $t$. It is possible that 
$a(t)=\boA(\Sigma)$ for $t\in (-\eps,\eps)$, $t\in[0,\eps)$ or $t\in(-\eps, 
0]$. 
In the first case, a whole neighborhood of $\Sigma$ is foliated by minimal 
hypersurfaces, we say that 
$\Sigma$ is \emph{minimally foliating on both side}. In both remaining cases, 
only one 
side of 
the neighborhood is foliated by minimal hypersurfaces, we say that $\Sigma$ is 
\emph{minimally foliating only on one side}.

If $\{S_t\}_t$ is such a smooth family of minimal surfaces with $\Sigma=S_0$, 
we notice that the derivative at $t$ of the family is given by a Jacobi field 
on $S_t$. So up to a change of parametrization $\{S_s\}_s$, we can assume that 
this Jacobi field as unit $L^2$-norm.

One can try to extend such a family. Let $\{S_s\}_{s\in (s^-,s^+)}$ be such 
a family of minimal hypersurfaces with $S_0=\Sigma$ (possibly $s^-=0$ and 
$s\in[0,s^+)$). First we notice that the hypersurfaces $S_s$ are disjoint for 
$s$ close to $0$. So if $S_s\cap S_{s'}\neq \emptyset$ for some $0\le s'<s$, 
there is a larger $\sigma$ such that all the hypersurfaces in $\{S_{s}\}_{0\le 
s<\sigma}$  are disjoint and $S_\sigma $ 
intersect $\Sigma$. Then the maximum principle implies that $\Sigma=S_\sigma$ 
and then $S_s=S_{s+\sigma}$ and the family is periodic and defined on $\R$.

Assume now that the family is made of disjoint minimal hypersurfaces. Each 
$S_s$ is 
an index $0$ minimal hypersurface of area $\boA(\Sigma)$. So by Sharp 
compactness result \cite{Shar} , as $s\to s^+$, a subsequence converges in 
the varifold 
sense to a minimal hypersurface $\Sigma^+$ : if $\Sigma^+$ is two-sided 
the convergence is smooth, if $\Sigma^+$ is one sided the convergence is with 
mulitplicity $2$ and smooth in the double cover of the tubular neighborhood of 
$\Sigma^+$. Actually this implies that as $s\to s^+$ the whole family $\{S_s\}$ 
converges to $\Sigma^+$. The same can be done as $s\to s^-$ to produce 
$\Sigma^-$. If $\Sigma^+$ is one sided, then the whole tubular neighborhood 
of $\Sigma^+$ is foliated by $\{S_s\}_{s\in (s^-,s^+]}$ where 
$S_{s^+}=\Sigma^+$. If 
$\Sigma^+$ is two sided, the family extends for $s=s^+$ and $S_{s^+}=\Sigma^+$ 
is either minimally foliating only on one side and the family 
$\{S_t\}_{t\in(s^-,s^+]}$ 
can't be extended across $s^+$ or $\Sigma^+$ is minimally foliating on both 
sides and  the family extends to $t\in(s^-,s')$ with $s^+<s'$. As a consequence 
we are in one of the following cases
\begin{enumerate}
\item the family extends to $\{S_s\}_{s\in [s^-,s^+]}$ and 
$S_{s^\pm}$ are minimally foliating only on one side or, in the case 
$s^-=0=s^+$, $\Sigma$ is 
not minimally foliating on any side
\item the family is periodic and gives a global foliation of $M$
\item the family extends to $\{S_s\}_{s\in(s^-,s^+]}$ such that 
$S_s\to \Sigma^-$ a $1$-sided minimal hypersurface as $s\to s^-$ and $S_{s^+}$ 
is minimally foliating only on one side.
\item the family extends to $\{S_s\}_{s\in(s^-,s^+)}$ such that $S_s\to 
\Sigma^\pm$ two $1$-sided minimal hypersurfaces as $s\to s^\pm$. In that case 
$\{S_s\}_{s\in[s^-,s^+]}$ with $S_{s_\pm}=\Sigma^\pm$ gives a global foliation 
of $M$.
\end{enumerate}
We notice that the family can't be defined on an unbounded interval without 
being periodic because of Sharp's compactness result and the fact that the 
family is parameterized at unit speed.

We then say that $\Sigma$ is \emph{almost stable} if, in the above cases $(1)$ 
and $(3)$, the area function $a^\pm$ associated to the canonical foliation 
given by $S_{s^\pm}$ also satisfies locally $a^\pm\ge 
\boA(S_{s^\pm})=\boA(\Sigma)$. We 
notice that, in case $(1)$ and $(3)$, $S_{s^\pm}$ is minimally foliating only 
on 
one side so there are value of $t$ close to $0$ such that 
$a^\pm(t)>\boA(\Sigma)$.

\subsection{Almost stable minimal hypersurfaces are local minima}

In the above classification we are interested in the situation of type 1 almost 
stable minimal hypersurface.

So we have a family of minimal hypersurface $\{S_s\}_{s\in[s^-,s^+]}$ that we 
can view as a subset $\boM$ of currents in $\boZ_n(M,\Z_2)$. We notice that 
$\boM$ is compact for the flat topology. In $\{S_s\}$, there are two particular 
element 
$S^-=S_{s^-}$ and $S^+=S_{s^-}$ that are minimally foliating only on one side; 
if 
$s^-=0=s^+$, 
$S^-=\Sigma=S^+$ and $\Sigma$ is not minimally foliating on any side. In the 
the non minimally foliated side, the leaves of the canonical foliation given by 
$S^\pm$ have area at least $\boA(\Sigma)$. However, it may contain minimal 
hypersurfaces of area $\boA(\Sigma)$, we add to 
$\boM$ all these minimal hypersurfaces that are sufficiently close to $S^\pm$ 
to define $\boM'$. $\boM'$ is still compact for the flat topology.

We have the following result

\begin{prop}\label{prop:locmin1}
Let $\Sigma$ be a type 1 almost stable minimal hypersurface and $\boM'$ be the 
set of 
minimal hypersurface of area $\boA(\Sigma)$ in the canonical foliation as 
defined above. Then there is an open set $U$ 
containing $\Sigma$ such that $\boA(\Sigma)\le \MM(T)$ for any 
$T\in\boZ_n(M,\Z_2)$ homologous to $\Sigma$ 
with support in $U$ and, if $\boA(\Sigma)= \MM(T)$, then $T\in \boM'$.
\end{prop}

\begin{proof}
The proof is similar to the one of \cite[Theorem 2]{Whi5}. So we consider the 
neighborhoods $K_r$, $0\le r\le \eps$ of $\Sigma$ given by \cite[Theorem 
1]{Whi5}. Let $T_r$ minimize the mass $\MM$ among all currents  homologous to 
$\Sigma$ in $K_r$. As $r\to 0 $, the currents $T_r$ must converge to $\Sigma$.

Since $T_r$ are uniformly almost minimizing, the convergence $T_r\to \Sigma$ 
and regularity theory implies that $T_r$ can be described as the normal graph 
of a function $u_r$ on $\Sigma$ with $u_r\to 0$ in $C^{1,\alpha}$. As $u_r$ is 
close to $0$, on can write $u_r= w_{t_r}+f_r$ where $w_t=tu_0+v_t$ is given by 
Lemma~\ref{lem:foliat} and $f_r\in V^{1,\alpha}$. So one can estimate the mass 
of $T_r$ by
\begin{align*}
\MM(T_r)&=A(u_r)\\
&=A(w_{t_r}+f_r)\\
&=A(w_{t_r})+\int_0^1DA(w_{t_r}+tf_r)(f_r)dt\\
&=A(w_{t_r})+DA(w_{t_r})(f_r)+\int_0^1\int_0^1tD^2A(w_{t_r}+stf_r)(f_r,f_r)dsdt\\
\end{align*}
By Lemma~\ref{lem:foliat} and \eqref{eq:differentiate}, $DA(w_{t_r})(f_r)=0$ 
and, for $r$ close to $0$, there is 
$c>0$ such that $D^2A(w_{t_r}+stf_r)(f_r,f_r)\ge c \|f_r\|_2^2$ since 
$D^2A(0)$ is given by the Jacobi operator. So 
$\MM(T_r)=A(w_{t_r})+\frac c2\|f_r\|_2^2\ge \boA(\Sigma)+\frac c2\|f_r\|_2^2$. 
Since $T_r$ is minimizing $\MM(T_r)\le \boA(\Sigma)$. So $f_r=0$ and 
$A(w_{t_r})=\boA(\Sigma)$ : $T_r=\Sigma_{w_{t_r}}$ which is minimal. So $T_r\in 
\boM'$ for $r$ close to $0$. 
\end{proof}

\begin{prop}\label{prop:locmin2}
Let $\Sigma$, $\boM$ and $\boM'$ be as above.
There is $\eps>0$ so that for every $T\in \boZ_n(M,\Z_2)$ with 
$\boF(T,\boM)<\eps$, we have $\MM(T)>\boA(\Sigma)$ unless $T\in \boM'$ where
$\MM(T)=\boA(\Sigma)$.
\end{prop}
\begin{proof}
The proof follows the lines of the proof of \cite[Proposition 6.2]{MaNe5}.

First let $(T_i)\in\boZ_n(M,\Z_2)$ be a sequence converging to some $S$ in 
$\boM$ in the flat topology with $\MM(T_i)\le \boA(\Sigma)$. From 
Proposition~\ref{prop:locmin1}, we obtain a neighborhood $U$ of $S$ such 
that elements in $\boM'$ are the only minimizers of $\MM$ among cycles 
contained 
in $U$ homologous to $S$. Let us consider a smaller neighborhood $V$ with 
$\barre 
V\subset U$. In \cite{MaNe5}, Marques and Neves then construct a sequence of 
cycles $R_i$ satisfying the following properties
\begin{itemize}
\item $\MM(R_i)\le \MM(T_i)\le \boA(\Sigma)$
\item $R_i\to S$ in the flat topology
\item the support of $R_i-T_i$ is contained in $M\setminus V$
\item the support of $R_i$ is in $U$ for large $i$
\end{itemize}
Since $R_i\to S$, $R_i$ is homologous to  $S$ for large $i$. So 
because of Proposition~\ref{prop:locmin1}, $R_i=S_i$ for some $S_i\in \boM'$. 
Moreover, since $R_i\to \Sigma$, $S_i\subset V$ for large $i$ then 
$\MM(T_i)=\MM(S_i+T_i\llcorner(M\setminus 
V))=\boA(S_i)+\MM(T_i\llcorner(M\setminus V))$ so $T_i\llcorner(M\setminus 
V)=0$ and $T_i=R_i=S_i$ \textit{i.e.} $T_i\in \boM'$.
\end{proof}

So the consequence of this is the following result.

\begin{prop}\label{prop:neigh}
Let $\Sigma$ and $\boM$ be as above.
For any $\eps>0$, there is $\delta>0$ and a neighborhood $\boN\subset 
\Z_n(M,\Z_2)$ of $\boM$, such that $\boN\subset 
\boN_\eps(\boM)=\{T\in\boZ_n(M,\Z_2)\mid \boF(T,\boM)<\eps\}$ and, for 
any $T\in \partial\boN$, we have $\MM(T)\ge 
\boA(\Sigma)+\delta$.
\end{prop}

\begin{proof}
First we assume that $\eps$ is such that Proposition~\ref{prop:locmin2} 
applies. Let $S_\pm$ denote both extremal hypersurfaces of the family (possibly 
$S^+=\Sigma=S^-$). Let $\{w_t^\pm\}=\{tu_0^\pm+v_t^\pm\}$ be the families of 
functions 
given by Lemma~\ref{lem:foliat} associated to $S^\pm$. Since $\Sigma$ is almost 
minimizing and $S^\pm$ are minimally foliating only on one side, there are 
$\delta>0$ and $t^+>0$ (resp. $t^-<0$) such that 
$S^\pm_{w_s^\pm}\in 
\boN_\eps(\boM)$  for $0\le s\le t^+$ (resp. $t^-\le s\le 0$) and 
$A(w_{t^\pm}^\pm)> \boA(\Sigma)$.

Let $\widetilde \boM=\boM\cup\{S_{w_s^\pm}^\pm,0\le \pm s\le t^\pm\}$. We can 
notice 
that any element in $\boM'\setminus \widetilde \boM$ is at a positive $\boF$ 
distance 
from $\widetilde \boM$. Let $\eta>0$ 
such that $\boN_\eta(\widetilde \boM)\subset \boN_\eps(\boM)$ and $\barre{ 
\boN_\eta(\widetilde \boM)}\cap \boM'=\widetilde \boM\cap \boM'$ . If $T_i$ is 
a 
sequence in 
$\partial \boN_\eta(\widetilde \boM)$ such that $\MM(T_i)\to \inf \{\MM(T),T\in 
\pr\boN_\eta(\widetilde \boM)\}$. By compactness we can assume that $T_i\to T$ 
and, by lower-semicontinuity of the mass, $\MM(T)= \inf \{\MM(T),T\in 
\pr\boN_\eta(\widetilde \boM)\}$. Moreover $\boF(T,\widetilde \boM)=\eta$ and 
$\boF(T,\boM)\le \eps$. So by Proposition~\ref{prop:locmin2}, $\MM(T)\ge 
\boA(\Sigma)$ 
with equality iff $T\in \boM'$ which would implies $T\in \widetilde \boM$ a 
contradiction. So $\MM(T)> \boA(\Sigma)$ and the result is proved.
\end{proof}

\subsection{The rigidity result}

In this section we finally prove Theorem~\ref{th:rigidityminmax}. So let us fix 
a Riemannian $3$-sphere $(\S^3,g)$ with sectional curvature bounded above by 
$1$. Its Simon-Smith width is realized by a collection of minimal spheres whose 
areas are at least $4\pi$ so the width is at least $4\pi$. If the width is 
$4\pi$, the width is then realized by a minimal $2$-sphere 
$\Sigma$ with multiplicity $1$. 


If $\Sigma$ has index at least $1$, the rigidity comes from 
Theorem~\ref{th:rigidity}. So 
we need to prove that index $0$ cannot occur. 

If $\Sigma$ has vanishing first Jacobi eigenvalue then $\Sigma$ may belong to a 
family of minimal spheres $\{S_s\}$. Since we are in $\S^3$, this family is of 
type 1. Let $\boM=\{S_s,s\in [s_-,s_+]\}$ be the associated compact subset in 
$\boZ_2(M,\Z_2)$ ($\boM=\{\Sigma\}$ if $\Sigma$ is stable). 

If $\Sigma$ is almost-stable (or $\Sigma$ stable), let $\eps$ be less that the 
$\boF$ distance from $\boM$ to the $0$ cycle and let $\boN$ be the neighborhood 
of $\boM$ given by Proposition~\ref{prop:neigh}.

Since $\Sigma$ realize the width, there are sequences $\{\sigma^n\}$ 
in $\Lambda$ and $(t_n)$ in $[-1,1]$ such that $\sigma_{t_n}^n\to 
\Sigma$ in the varifold sense and $(\sigma_{t_n}^n)$ is a min-max sequence 
which is almost minimizing in small annuli. Because of 
Proposition~\ref{prop:currentlimit}, we have $\sigma_{t_n}^n\to 
\Sigma\in \boZ_2(M,\Z_2)$ for the $\boF$ metric. So for $n$ large enough, the 
continuous path $t\mapsto \sigma_t^n$ in $\boZ_2(M,\Z_2)$ enters into $\boN$. 
As this path starts and ends at the $0$ cycle, it must cross $\partial\boN$ so 
\[
\max_{t\in[-1,1]}\boA(\sigma_t^n)\ge 4\pi +\delta
\]
where $\delta$ is given by 
Proposition~\ref{prop:neigh}. So $\{\sigma^n\}$ can not be a minimizing 
sequence.

Thus the width is realized by a minimal sphere $\Sigma$ in $M$ of area $4\pi$ 
with vanishing first Jacobi eigenvalue and which is not almost-stable. $\Sigma$ 
belongs to a local foliation by minimal surfaces (it may contain only one leaf) 
which contain one minimal sphere $\barre S$ of area $4\pi$ whose canonical 
foliation 
$\{S_t\}$ contains leaves such that $\boA(S_t)<\boA(\Sigma)$ for $t>0$ 
arbitrarily close to $0$. Notice that $\barre S$ has also vanishing first 
Jacobi eigenvalue. Let $a(t)=\boA(S_t)$ be the associated area function. 
We notice that, if $a'(t)=0$, $S_t$ is a minimal sphere and then $a(t)\ge 
4\pi$. This implies that
\begin{itemize}
\item either $a(t)\ge 4\pi$ for $t\in[0,\eps)$,
\item or $a(t)$ is decreasing on $[0,\eps)$.
\end{itemize}
Since $\barre S$ is not almost-stable, we are in the second situation. In order 
to exploit this situation we need to introduce a slightly different local 
foliation near $\barre S$.
\begin{lem}
Let $\Sigma$ be a minimal $2$-sided hypersurface with vanishing first jacobi 
eigenvalue. There is $\eps>0$ and a smooth map 
$\tilde v:(-\eps,\eps)\times \Sigma\to \R$ with the following properties.
\begin{itemize}
\item $\tilde v_0=0$, ${\pr_t \tilde v_t}_{|t=0}=0$ and $\int_\Sigma u_0\tilde 
v_t=0$.
\item for each $t\in(-\eps,\eps)$, $H_{tu_0+\tilde v_t}$ is a multiple of $u_0$.
\end{itemize}
\end{lem}
The proof is the same as Lemma~\ref{lem:foliat} proof and gives the same 
consequence: if $\Sigma_w$ is minimal hypersurface with $w$ small 
$w=tu_0+\tilde v_t$ for some $t$.

So let $\{\widetilde S_t\}=\{\barre S_{tu_0+\tilde v_t}\}$ be  the family given 
by 
applying the  above lemma to $\barre S$. Let $\tilde 
a(t)=\boA(\widetilde{S}_t)$ be the associated area function. If $\tilde 
a'(t)=0$, 
$\widetilde S_t$ is a minimal sphere 
and then $\tilde a(t)\ge 4\pi$. So we also have the alternative
\begin{itemize}
\item either $\tilde a(t)\ge 4\pi$ for $t\in[0,\eps)$,
\item or $\tilde a(t)$ is decreasing on $[0,\eps)$.
\end{itemize}
Let us see that we are in the second case. If we are in the first one, there is 
$t_0>0$ such that $\tilde a'(t_0)\ge 0$, this implies that the mean curvature 
vector of $\widetilde S_{t_0}$ points to $\barre S$. For $t>0$ small, $S_t$ is 
between $\bar S$ and $\widetilde S_{t_0}$, let $t_1$ be the first $t$ such that 
$S_t$ 
intersect $\widetilde S_{t_0}$, then the mean curvature vector of $S_{t_1}$ 
must points to $\barre S$ which contradict $a'(t_1)<0$. So we are in the second 
case. 

If $\tilde w_t=tu_0+\tilde v_t$, we define on the neighborhood of $\barre S$ 
the vectorfield $X$ by $X(\Phi_{\tilde w_t}(p))=(\pr_t\tilde w_t)\pr_t$ such 
that $\widetilde S_t$ is the image of $\barre S$ by the flow given by $X$. We 
also consider $\tilde \nu_t$ the unit normal to $\widetilde S_t$ and 
$Y=(X,\tilde \nu_t)\tilde \nu_t$. This vectorfield is normal to 
$\widetilde S_t$ and $\widetilde{S}_t$ is the image of $\barre S$ by the flow 
given by $Y$. So looking at the $F$ functional on $\widetilde S_t$ we have:
\begin{align*}
\frac{d}{dt} F(\widetilde S_t)&=-\int_{\widetilde S_t} 2\widetilde H_t  (Y, 
\tilde \nu_t)d\sigma -\int_{\widetilde S_t}2\widetilde H_t^3 
(Y,\tilde \nu_t)d\sigma\\
&\qquad +\int_{\widetilde S_t} (\Delta^\perp Y,H_t\tilde \nu_t)+( 
Y,B(e_i,e_j))(B(e_i,e_j),\widetilde 
H_t\tilde \nu_t) d \sigma\\
&\qquad+\int_{\widetilde S_t} \widetilde 
H_tR(e_i,Y)e_i,\tilde \nu_t)d \sigma\\
&=\int_{\widetilde S_t} \widetilde H_t(X,\tilde \nu_t) 
(\ric(\tilde \nu_t)-2)d\sigma+\int_{\widetilde S_t} 
\widetilde H_t(X,\tilde \nu_t)\|\inter B\|^2d \sigma\\
&\qquad-\int_{\widetilde S_t}(\nabla (X,\tilde\nu_t),\nabla 
\widetilde H_t)d \sigma
\end{align*}
where $\Delta^\perp$ and $\nabla$ are operators on $\widetilde S_t$ and 
$\widetilde H_t$ is the mean curvature of $\widetilde S_t$.

Let $\tilde u_t$ be the positive function defined on $\widetilde S_t$ by 
$\tilde u_t(\Phi_{\tilde w_t}(p))=u_0(p)$. By construction there is $\tilde 
c_t\in \R$ 
such that $\widetilde H_t=\tilde c_t \tilde u_t$. Moreover since we are in the 
second case, $\tilde c_t>0$ and $\widetilde H_t>0$. This also implies 
$\widetilde H_t(X,\tilde \nu_t) \ge0$ and $(\nabla (X, \tilde \nu_t),\nabla 
\widetilde H_t)=c_t(\nabla (X,\tilde \nu_t),\nabla \tilde u_t)$. At $t=0$, 
$(\nabla (X,\tilde \nu_t),\nabla \tilde u_t)=\|\nabla u_0\|^2\ge 0$. So if 
$u_0$ is not constant 
$\int_{\widetilde S_t}(\nabla (X,\tilde \nu_t),\nabla \widetilde H_t)d 
\sigma\ge 0$ 
for small $t$ and, if $u_0$ is constant, $\widetilde{H}_t=\tilde c_t\tilde u_t$ 
is constant and $\int_{\widetilde S_t}(\nabla (Y,\tilde \nu_t),\nabla 
\widetilde H_t)d \sigma$ vanishes. In both cases
\[
\frac{d}{dt} F(\widetilde S_t)\le CF(\widetilde S_t)
\]
for some constant $C$.
This implies that $F(\widetilde S_t)=0$ for any $t$. So we are in the equality 
case: $\ric(\nu_t)=2$ since $c_t>0$. So $\ric(\nu_0)=2$ and the Jacobi operator 
on $\barre S$ is $\Delta+2$ which contradict
that $0$ is its first eigenvalue.

\appendix

\section{A Schauder estimate}\label{app:schauder}

In this appendix, we prove Proposition~\ref{prop:schauder}. To lighten 
notations, we denote $E_T=E_{[0,T]}$ and 
$P_T=\R^n\times[0,T]$. First, we give some 
complementary notations for H\"older norms. For a map $u:D\subset \R^n\times 
\R\to \R^k$, we define
\begin{gather*}
[u]_{2,\alpha,D}=[u]_{2,\alpha,D,x}+[u]_{2,\alpha,D,t}\\
[u]_{1,\alpha,D,x}=[\pr_x u]_{0,\alpha,D,x}\\
[u]_{1,\alpha,D,t}=[u]_{(1+\alpha)/2,D,t}+[\pr_x u]_{\alpha/2,D,t}\\
[u]_{1,\alpha,D}=[u]_{1,\alpha,D,x}+[u]_{1,\alpha,P,t}\\
\|u\|_{1,\alpha,D}=\|u\|_{0,D}+\|\pr_x u\|_{0,D}+[u]_{1,\alpha,D}
\end{gather*}

Associated to these norms, we have the $C^{l,\alpha}$ H\"older space. We also 
define the space $C_0^{l,\alpha}(P_T)$ as maps $u\in C^{l,\alpha}(P_T)$ 
satisfying $\pr_t^i u_{|t=0}=0$ for $0\le i\le l/2$.

We have some classical interpolation inequalities. 
\begin{lem}
Let $l,m\in\{0,1,2\}$, $\alpha,\beta\in[0,1]$ such that $l+\alpha> m+\beta$ and 
let $T>0$. 
Then for any $\eps>0$ there is a constant $C$ 
such that, for any $u:P_T\to \R^k$
\[
\|u\|_{m,\beta,P_T}\le C\|u\|_{L^2(P_T)}+\eps [u]_{l,\alpha,P_T}
\] 
\end{lem}
\begin{proof}
Let us explain the proof when $l=m=0$ and $\beta=0<\alpha$. Let us notice that 
its sufficient to look only at one component $u^a$ of $u$. Let $X\in P_T$ and 
consider $\delta>0$. For $\delta$ small, we can consider a box $B$ which is a 
translate of $[0,\delta]^n\times[0,\delta^2]$ such that $B\subset P_T$ and 
$X\in 
B$ ($\delta$ can be chosen independently of $X$). Then there is $\barre X\in B$ 
such that $\int_B u^a=\delta^{n+2}u^a(\bar 
X)$. So 
\begin{align*}
|u^a(X)|&\le |u^a(\barre X)|+|u^a(\barre X)-u^a(X)|\\
&\le \frac{1}{\delta^{n+2}}\int_B |u^a|+2\delta^\alpha[u^a]_{0,\alpha,B}\\
&\le \frac{1}{\delta^{n/2+1}}\|u^a\|_{L^2(B)}+2\delta^\alpha[u^a]_{0,\alpha,B}\\
& \le 
\frac{1}{\delta^{n/2+1}}\|u^a\|_{L^2(P_T)}+2\delta^\alpha[u^a]_{0,\alpha,P_T}\\
\end{align*}
So choosing $\delta$ small enough we have the result.

Once this first estimate is established, the other ones can be obtained by 
similar arguments (for example, see Section~6.8 in \cite{GiTr}).
\end{proof}

The second result that we shall need is a Schauder type estimate for solution 
of 
\begin{equation}\label{eq:parab}
\pr_tu-\boL u=f
\end{equation}
over $\R^n\times[0,T]$ where $\boL$ is an operator as in \eqref{eq:operator} 
with constant coefficients and only second order terms.

\begin{lem}
Let $\boL$ be an operator as in \eqref{eq:operator} with constant coefficients 
and only second order terms. 
Then there is a constant $C$ such that the following statement is true. If 
$u\in C_0^{2,\alpha}(P_T)$ is a solution of \eqref{eq:parab} with $f\in 
C_0^{0,\alpha}(P_T)$ such that $u_t$ has compact support for any $t$ then
\begin{equation}\label{eq:schauder1}
[u]_{2,\alpha,P_T}\le C[f]_{0,\alpha,P_T}
\end{equation}
\end{lem}
This result is established in Section~15 of \cite{Slk} (see Theorem~4.1 and 
Equation~(4.43))

We now want a similar result when $L$ depends on the variable $x$ and have 
terms of any order.
\begin{lem}\label{lem:schauder}
Let $L$ be an operator as in \eqref{eq:operator} with $C^{\alpha}$ coefficients 
independent of $t$. 
Then there is a constant $C$ such that the following statement is true. If 
$u\in C_0^{2,\alpha}(P_T)$ is a solution of \eqref{eq:parab} with $f\in 
C_0^{0,\alpha}(P_T)$ such that $u_t$ has compact support for any $t$ then
\begin{gather*}
[u]_{2,\alpha,P_T}\le C(\|f\|_{0,\alpha,P_T}+\|u\|_{L^2(P_T)}\\
\|u\|_{2,\alpha,P_T}\le C (\|f\|_{0,\alpha,P_T} +\|u\|_{L^2(P_T)})
\end{gather*}
\end{lem}

\begin{proof}
Let $p$ be a point in $\R^n$, we are going to prove the estimate near $p$. Let 
$\delta>0$ and consider $\phi_\delta$ be a non-negative $C^\infty$ function on 
$\R^n$ with support on the ball $B$ centered at $p$ and radius $2\delta$, equal 
to $1$ on the ball $B'$ of 
radius $\delta$ and such that $[\phi_\delta]_{i,E}\le \frac C{\delta^i}$. We 
are going to estimate 
$\phi_\delta u$.

Let $L_p$ denote the operator $L(p)$ and $\boL_p$ the part of $L_p$ with only 
second order terms. Since $\pr_t u-Lu=f$ we have
\[
\pr_t (\phi_\delta u)-\boL_p(\phi_\delta u)=\phi_\delta 
f+\phi_\delta Lu-L(\phi_\delta u)+L(\phi_\delta u)-L_p(\phi_\delta u) 
+L_p(\phi_\delta u)-\boL_p(\phi_\delta u)
\]

So the estimate~\eqref{eq:schauder1} gives 
\begin{align*}
[\phi_\delta u]_{2,\alpha,B_T}\le& C([\phi_\delta 
f]_{0,\alpha,B_T}+[\phi_\delta 
Lu-L(\phi_\delta u)]_{0,\alpha,B_T}\\
&\qquad+[L(\phi_\delta u)-L_p(\phi_\delta 
u)]_{0,\alpha,B_T}+[L_p(\phi_\delta u)-\boL_p(\phi_\delta u)]_{0,\alpha,B_T})
\end{align*}
where $B_T$ denotes $B\times [0,T]$.

In the right hand side of the above estimate, the first term can be estimated 
by 
$[\phi_\delta f]_{0,\alpha,B_T}\le C_\delta \|f\|_{0,\alpha,B_T}$ (in the 
sequel $C_\delta$ will denote a constant that depends on $\delta$). In 
$\phi_\delta Lu-L(\phi_\delta u)$, the terms where the second derivatives of 
$u$ appears cancel, so $[\phi_\delta 
Lu-L(\phi_\delta u)]_{0,\alpha,B_T}\le C_\delta \|u\|_{1,\alpha,B_T}$. 
Similarly for the last term, $[L_p(\phi_\delta u)-\boL_p(\phi_\delta 
u)]_{0,\alpha,B_T}\le C_\delta \|u\|_{1,\alpha,B_T}$. For the 
third term, if $\Lambda$ bounds the $C^{0,\alpha}$ norm of the coefficients 
of $L_p$, we have
\[
[L(\phi_\delta u)-L_p(\phi_\delta u)]_{0,\alpha,B_T}\le 
\Lambda\delta^\alpha[\phi_\delta u]_{2,\alpha,B_T}+2\Lambda\|\phi_\delta 
u\|_{2,B_T}
\]

So fixing $\delta$ such that $C\Lambda\delta^\alpha<1$ we obtain
\[
[\phi_\delta u]_{2,\alpha,B_T}\le C_\delta (\|f\|_{0,\alpha,B_T} 
+\|u\|_{1,\alpha,B_T} +\|\phi_\delta u\|_{2,B_T})
\]
Since $\phi=1$ on $B'$ we obtain
\[
[u]_{2,\alpha,B'_T}\le C_\delta (\|f\|_{0,\alpha,P_T} 
+\|u\|_{1,\alpha,P_T} +\| u\|_{2,P_T})
\]
Since we can consider any point $p$ in $\R^n$, we have
\[
[u]_{2,\alpha,P_T}\le C_\delta (\|f\|_{0,\alpha,P_T} 
+\|u\|_{1,\alpha,P_T} +\| u\|_{2,P_T})
\]

By interpolation inequalities, we have $\|u\|_{1,\alpha,P_T} +\| u\|_{2,P_T}\le 
C_\eps \|u\|_{L^2(P_T)}+\eps [u]_{2,\alpha,P_T}$. So choosing $\eps$ such that 
$C_\delta \eps<1$, we obtain 
\[
[u]_{2,\alpha,P_T}\le C_\delta (\|f\|_{0,\alpha,P_T} +\|u\|_{L^2(P_T)})
\]
Using interpolation inequality again we finally have
\[
\|u\|_{2,\alpha,P_T}\le C_\delta (\|f\|_{0,\alpha,P_T} +\|u\|_{L^2(P_T)})
\]
\end{proof}

We can now give the proof of Proposition~\ref{prop:schauder}.

\begin{proof}[Proof of Proposition~\ref{prop:schauder}]
Let $U$ and $F$ be as in the proposition. If $T\le 1$ the result is given by 
Theorem~\ref{thm:schauder}, so let us assume that $T\ge 1$ and let $\eta 
:\R_+\to \R_+$ such that $\eta=1$ in a neighborhood of $0$ and $\eta=0$ on 
$[1/2,+\infty)$. So we can write $U=\eta(t) U+(1-\eta(t)) U$. By 
Theorem~\ref{thm:schauder}, we have $\|\eta(t)U\|_{2,\alpha,E_T}\le 
C\|U\|_{2,\alpha,E_1}\le 
C(\|F\|_{0,\alpha,E_T}+|U_0|_{2,\alpha,E}+\|U\|_{L^2(E_T)})$. The 
section $W=(1-\eta)U\in C_0^{2,\alpha}(E_T)$ is then a solution of 
\[
\pr_t W-L W=(1-\eta)F-\eta' U
\]

Let $(\phi_i)_{1\le i\le m}$ be a partition of the unity such that the $\phi_i$ 
has support in local charts $\Ome^i$ of $E$. Then each $W_i=\phi_i W$ is a 
solution of $\pr_tW_i-LW_i=G_i$. In the chart, $W_i$ can be written as 
a section $w_i$ over $\R^n$ with compact support which is a solution of $\pr_t 
w_i-Lw_i=g_i$. So, by Lemma~\ref{lem:schauder}
\[
\|w_i\|_{2,\alpha,P_T}\le C(\|g_i\|_{0,\alpha,P_T}+\|w_i\|_{L^2(P_T)})
\]
This gives
\[
\|W_i\|_{2,\alpha,{\Ome^i}_T}\le 
C(\|g_i\|_{0,\alpha,{\Ome^i}_T}+\|w_i\|_{L^2({\Ome^i}_T)})
\]
where ${\Ome^i}_T$ denotes the bundle over $\Ome^i\times [0,T]$.
Summing these estimates and using that a finite number of $W_i$ is sufficient 
we obtain
\begin{align*}
\|W\|_{2,\alpha,1,E_T}&\le 
C(\|F\|_{0,\alpha,E_T}+\|U\|_{0,\alpha,E_1}+\|U\|_{L^2(E_T)})\\
&\le C(\|F\|_{0,\alpha,E_T}+|U_0|_{2,\alpha,E}+\|U\|_{L^2(E_T)})\\
\end{align*}
where we have used $\eta'=0$ outside $[0,1]$ and Solonnikov's estimate on 
$E_1$. So adding both estimates for $\eta U$ and $(1-\eta)U$ gives the desired 
result.
\end{proof}

\end{document}